\documentclass[journal]{IEEEtran}
\usepackage{tikz}
\usepackage{pgf}

\usetikzlibrary{arrows,automata}
\usetikzlibrary{matrix,arrows,decorations.pathmorphing}

\usepackage{tkz-berge}

\usetikzlibrary{arrows,shapes}
\usepackage{amssymb}

\def\0{{\bf 0}}
\def\1{{\bf 1}}
\def\Sigmah{\widehat{\Sigma}}

\ifCLASSINFOpdf
\else
\fi

\usepackage {amsmath}
\usepackage {amsthm}

\usepackage{url}

\usepackage{graphicx}

\newtheorem{theorem}{Theorem}

\newtheorem{corollary}[theorem]{Corollary}

\newtheorem{lemma}[theorem]{Lemma}

\newtheorem{proposition}[theorem]{Proposition}

\newcommand {\R} {\mathbb {R}}

\def\1{{\bf 1}}

\newcommand{\x}{{\bf x}}

\newcommand{\p}{{\bf p}}

\newcommand{\specialcell}[2][c]{%
  \begin{tabular}[#1]{@{}c@{}}#2\end{tabular}}

\begin{document}
%
\title{Scaling laws for consensus protocols subject to noise}
%
%
%

\author{Ali~Jadbabaie\thanks{Institute for Data, Systems, and Society, MIT, {\tt jadbabai@mit.edu}}, Alex~Olshevsky\thanks{Department of Electrical and Computer Engineering and the Division of Systems Engineering, Boston University, {\tt alexols@bu.edu}}
\thanks{This research was supported by Air Force award FA990-15-10394 and the AFOSR Complex Networks Program.}}

\maketitle

\begin{abstract}  We study the performance of discrete-time consensus protocols in the presence of additive noise. When the consensus dynamic corresponds to a reversible Markov chain,  we give an exact expression for a weighted version of steady-state disagreement in terms of the stationary distribution and hitting times in an underlying graph. We then show how this result can be used to characterize the noise robustness of a class of protocols for formation control in terms of the Kemeny constant of an underlying graph.

\end{abstract}

%
\IEEEpeerreviewmaketitle

\section{Introduction} The design of policies for control and signal processing in networks of agents (such as UAVs, vehicles, or sensors) has attracted considerable attention over the past  decades. It is commonly believed that such policies benefit from being distributed, for example by relying only on local, nearest-neighbor interactions in a network of nodes.  Understanding how simple, distributed updates can accomplish global objectives and giving quantifiable bounds on their performance has correspondingly been an active area of research recently.

An emerging understanding is that a key tool for such systems is the so-called consensus iteration, namely the update  \[ x(t+1) = P x(t),\] where $P$ is a stochastic matrix. This update has the property that, subject to some technical assumptions on the matrix $P$, the vector $x(t)$ converges to ${\rm span} \{ \1 \}$, the subspace spanned by the all-ones vector. One usually thinks of each component $x_i(t)$ as being controlled by a different ``agent,'' with the agents asymptotically ''coming to consensus'' as  all the components of $x(t)$ approach the same value. 

 It turns out that many sophisticated network coordination tasks can be either entirely reduced to consensus or have decentralized solutions where the consensus iteration plays a key role; we mention formation control \cite{ofm07, ren-form1, ren-form2, othesis}, distributed optimization \cite{tsi, NO}, coverage control \cite{gcb08, conscov2}, distributed task assignment \cite{CBH09, cons-task2},  networked Kalman filtering  \cite{kalman1, kalman2, kalman3, kalman4},  cooperative flocking/leader-following \cite{JLM, leader2}, among many others. 

For example, it is wel-known (and we spell out later in this paper) that the problem of maintaining a formation when every agent can measure the relative positions of its neighbors can be solved with a distributed update rule which turns out to be the consensus update after a change of variable.

As a consequence of the many  applications of consensus, a large literature has recently built up around it. In this paper, we contribute to a strand of this literature which studies the effect of noise; specifically, we study the noisy consensus iteration
 \begin{equation} \label{noisecons} x(t+1) = P x(t) + w(t), \end{equation}  where the matrix $P$ is stochastic as before and the vector 
$w(t)$ represents zero-mean i.i.d. noise. Our goal in the present paper is to contribute to an understanding of how much the ``coming to consensus'' property deteriorates due to the addition of the noise term $w(t)$ in Eq. (\ref{noisecons}). The main concern of this paper is the notion of expected disagreement; namely we will consider the average expected square deviation of the values $x_i(t)$ from their (weighted) average as $t \rightarrow +\infty$. 

Intuitively, the action of multiplying a vector $x(t)$ by a stochastic matrix $P$ has the effect of bringing the components $x_i(t)$ ``closer together,'' while the addition of the noise $w(t)$ counteracts that; the two processes might be expected to balance in some equilibrium level of expected disagreement as $t \rightarrow \infty$. We are motivated by the observation (made in the previous literature on the subject and discussed extensively later) that the equilibrium level of disagreement often grows with the dimension of the matrix $P$. 

Thus even though Eq. (\ref{noisecons}) can be shown to be stable (under some mild technical assumptions) in the sense that expected disagreement between any pair of nodes is bounded as $t \rightarrow \infty$, this stability can be almost meaningless for many classes of large systems in which expected disagreement scales with dimension. This has implications for all distributed protocols which rely on consensus, as it implies that in some cases they may lack robustness under the addition of noise. Building on previous work and contributing to the study of this phenomenon is the goal of the present paper.


%


\subsection{Literature review\label{prevwork}} The problem of analyzing the steady-state level of disagreement in consensus with noise was initiated in \cite{boyd} where, for a symmetric matrix $P$ and under the assumption that the components of $w(t)$ are uncorrelated, an explicit expression in terms of the eigenvalues of $P$ was given.  Recently \cite{zmp} gave an alternative expression in terms of the average resistance associated with a graph based on the symmetric matrix $P$, and further showed this expression can be used to bound the steady-state disagreement from above and below in the more general case when the stochastic matrix $P$ is not necessarily symmetric but rather corresponds to a reversible Markov chain. 

Continuous analogues of Eq. (\ref{noisecons}) have also been studied. When the underlying dynamics comes from a symmetric graph Laplacian,  expressions for equilibrium disagreement in terms of eigenvalues, resistances, and hitting times were presented in \cite{bamieh2}. When the underlying dynamics is not necessarily symmetric but satisfies a normality property, expressions for disagreement in terms of eigenvalues were given in \cite{naomi1}.

The observation that Eq. (\ref{noisecons}) can have asymptotic disagreement which grows with the size of the system was, to our knowledge, first made in \cite{bamieh1} (in continuous time). As observed in \cite{bamieh1} in the context of vehicular formation control, this means that any protocol which relies on consensus iterations 
can suffer from a considerable degradation of performance in large networks. Furthermore, \cite{bamieh1} showed that topology can have a profound influence on performance by proving that while on the ring 
graph the asymptotic disagreement grows linearly with the number of nodes, it remains bounded on the 3D torus (and grows only logarithmically in the number of nodes on the 2D torus).

We next mention some other  related strands of literature. The paper \cite{mot1} which investigated consensus-like protocols with noise in continuous time, focusing on connections with measures of sparsity such as number of spanning trees, number of cut-edges, and the degree-sequence. The related paper \cite{mot2} investigated several  measures of robustness related to equilibrium disagreement in terms of their convexity.  The paper \cite{bamieh2} characterized steady-state disagreement in a number of fractal graphs.
The recent papers \cite{naomi1, naomi2} considered the effects of noise in a continuous-time version of Eq. (\ref{noisecons}) over {directed} graphs. In \cite{naomi1}, explicit expressions for a measure of steady-state disagreement were computed for a number of such graphs. The paper \cite{naomi2} investigated steady-state disagreement on trees and derived a partial ordering capturing which trees have smaller steady-state disagreements. In \cite{naomi3}, noisy consensus with a drift term was considered with a focus on ranking nodes in terms of their variance growth. Our earlier work \cite{jo-prev} focused on connections between asymptotic disagreement and the Cheeger constant and coefficients of ergodicity of the corresponding Markov chain. Moreover, we mention the recent paper \cite{summers} which considered approximation algorithms for the problem of designing networks that minimize equilibrium disagreement. Finally, there is considerable work on the  leader selection problem, where only some of the nodes are performing consensus iteration, which has a similar flavor and which we do not survey here.

\subsection{Our contributions} Under the assumption that the matrix $P$ is reversible, we give an {\em exact} expression for a weighted version of the equilibrium disagreement where the disagreement at each node is weighted proportionally to its importance in the graph corresponding to the matrix $P$. Our expression is combinatorial in that it involves hitting times and the stationary distribution of $P$. Furthermore, we allow the noise $w(t)$ to have any covariance matrix (though it must be i.i.d. in time). In the previous literature such expressions were available only for the special case when the matrix $P$ was symmetric and all the noises $w_i(t)$ were uncorrelated. 

This expression is the main result of this paper and it has three main consequences. First, our expression allows us to compute the weighted steady-state disagreement corresponding to simple averaging on undirected graphs, when each node puts an equal weight on all its neighbors. Updates of this form are the canonical example of distributed averaging algorithms. As a consequence, we are able to compute the weighted steady state disagreement for such updates on many graphs, ranging from simple examples such as the line graph and the star graph, to more sophisticated cases such as Erdos-Renyi random graphs and dense regular graphs.

Secondly, our results lead to an explicit combinatorial expression (again in terms of hitting times and the stationary distribution of $P$) which provides the best known approximation for the unweighted steady-state disagreement (where the disagreement of each node is weighted equally). This improves on the results of \cite{zmp}, which had the previously best combinatorial approximation (in terms of graph resistances) of unweighted disagreement.

Thirdly, this generalization allows us to apply our results to the problem of formation control and characterize the noise resilience of a natural class of first-order protocols for it. It turns out that there is a very simple expression for the noise resilience of a symmetric formation control protocol: we show it is proportional to the so-called Kemeny constant of an underlying graph. This observation is new and allows for the easy computation of the scaling of noise resilience on a variety of graphs. 

Finally, we remark that our proof strategy is also of independent interest on its own. Most previous papers relied on explicit diagonalization of the system of Eq. (\ref{noisecons}). This can be made to work when the eigenvalues of $P$ are known and $P$ is symmetric (allowing us to change variables without affecting the covariance of the noise $w(t)$). However, this approach runs into obstacles when either of these assumptions is not satisfied. Here we introduce a different idea: we relate the recursions for steady-state covariance of Eq. (\ref{noisecons}) to recursions for hitting times on certain graphs.

\subsection{The organization of this paper} 

The main result of this paper, namely an exact expression for weighted steady-state disagreement in noisy consensus, is presented in Section \ref{proof} as Theorem \ref{mainthm}. Section \ref{proof} gives a proof of this result and then discusses simplifications in a number of special cases. Additionally bounds on unweighted steady-state disagreement are presented which improve on the current state of the art from \cite{zmp}. 

Section \ref{examples} then uses Theorem \ref{mainthm} to work out how disagreement scales with the number of nodes for simple averaging on a number of common graphs. Section \ref{symmetric}  introduces the problem of understanding the performance of formation control from noisy measurements of relative position, and, using Theorem \ref{mainthm}, characterizes this  in terms of the Kemeny constant of an underlying graph. Finally, simulations are provided in Section \ref{simul} and the paper finishes with some concluding remarks in Section \ref{conclusion}.

\section{Asymptotic disagreement in noisy consensus\label{proof}} In this section, we state and prove our main result, which is an explicit expression for the weighted steady state disagreement in noisy consensus. Additionally, we work some simplifications of our result for the case when the matrix $P$ from Eq. (\ref{noisecons}) is symmetric and discuss connections to resistance, Kemeny constant, and unweighted steady-state disagreement. 

We begin with a concise statement of our main result, starting with a number of definitions. {\em For the remainder of this paper, we will assume $P$ to be a stochastic, irreducible, and aperiodic matrix}, and we let $\pi$ be the stationary distribution vector of the Markov chain with transition matrix $P$, i.e., $$\pi^T P = \pi^T, ~~~~~~~~~ \sum_{i=1}^n \pi_i = 1.$$ Alternatively, $\pi$ is simply the unique normalized left-eigenvector corresponding to the dominant eigenvalue of $1$ of the stochastic matrix $P$. We note that, for the remainder paper, we will find it convenient to abuse notation by conflating the matrix $P$ and the Markov chain whose transition matrix is $P$ (for example, we will refer to $\pi$ as the stationary distribution of $P$). 

We will use $D_{\rm \pi}$ to stand for the diagonal matrix whose $i$'th diagonal entry is $\pi_i$.  
Furthermore, we define the weighted average vector, \vspace{0.1cm} $$ \overline{x}(t) := \left( \sum_{i=1}^n \pi_i x_i(t) \right) \1,$$ as well as the error vector $$e(t) := x(t) - \overline{x}(t).$$

\vspace{0.1cm}
 Intuitively, $e(t)$ measures how far away the vector $x(t)$ is from consensus. Indeed, it is easy to see that the noiseless update $x(t+1)=Px(t)$ has the property that $x(t)$ converges to $\left( \sum_i \pi_i x_i(0) \right) \1$. The quantity $e(t)$ thus measures the difference between the ``current state'' $x(t)$ and the limit of the noiseless version of Eq. (\ref{noisecons}) starting from $x(t)$. 

Our  goal is to understand how big  the error $e(t)$ can get as $t$ goes to infinity. Due to the addition of noise $w(t)$ in Eq. (\ref{noisecons}), the error vectors $e(t)$ are random variables. Recall that $w(t)$ is zero-mean  i.i.d., and we now  introduce the notation $\Sigma_{w}$ for its covariance.  In order to measure deviation from consensus, we will consider  the following two linear combinations of squared errors at each node,  \vspace{-0.3cm} \begin{eqnarray*} \delta(t) & := & \sum_{i=1}^n \pi_i E[ e_i^2(t)] \\
\delta^{\rm uni}(t) & := & \frac{1}{n} \sum_{i=1}^n E[e_i^2(t)],
\end{eqnarray*}  i.e., we weigh the squared error at each node  either proportionally to the stationary distribution of the node or uniformly.  Finally, our actual measures of performance will be the asymptotic quantities  \vspace{0.01cm}
\vspace{-0.01cm} \begin{eqnarray*} \delta_{\rm ss} & := & \lim \sup_{t \rightarrow \infty} \delta(t) \\
\delta_{\rm ss}^{\rm uni} & := & \lim \sup_{t \rightarrow \infty} \delta^{\rm uni}(t),
\end{eqnarray*} which capture the limiting disagreement among the nodes. We will refer to these quantities as weighted steady-state disagreement and unweighted steady-state disagreement, respectively. We will sometimes write $\delta_{\rm ss}(P, \Sigma_{w})$ when the update matrix $P$ and the noise covariance $\Sigma_w$ are not clear from context and likewise for $\delta_{\rm ss}^{\rm uni}$.

Before stating our main result, let us recall the notion of hitting time from node $i$ to node $j$ in a Markov chain: this is the expected time until the chain visits $j$ for the first time starting from node $i$.  We use $H_M(i \rightarrow j)$ to denote this hitting time in the Markov chain whose probability transition matrix is $M$.  By convention, $H_M(i \rightarrow i)=0$ for all $i$. We will use the notation $H_M$ to denote the matrix whose $i,j$'th element is $H_M(i \rightarrow j)$. For a comprehensive treatement of hitting times, the reader may consult the recent textbook \cite{lpw}.

With the above definitions in place, our next theorem contains our main result on steady-state disagreement. We remind the reader that, in addition to the stated assumptions of the theorem, we are also assuming that $P$ is a stochastic, irreducible, and aperiodic matrix for the remainder of the paper.

\bigskip

\begin{theorem} (An Explicit Expression for Weighted Steady-State Disagreement) If the Markov chain with transition matrix $P$ is reversible, then \[ \delta_{\rm ss}(P, \Sigma_w)  =  \pi^T H_{P^2} D_{\rm \pi} \Sigma_w D_{\rm \pi} \1 - {\rm Tr}  (H_{P^2}  D_{\rm \pi} \Sigma_{w} D_{\rm \pi}). \] \label{mainthm}
\end{theorem}  The theorem characterizes $\delta_{\rm ss}$ in terms of combinatorial quantities associated with an underlying Markov chain, namely the stationary distribution and the hitting times. Inspecting the theorem, we note that it expresses $\delta_{\rm ss}$ in terms of a difference of two linear combinations of entries of the matrix $H_{P^2} D_{\pi} \Sigma_w D_{\pi}$, both with nonnegative coefficients which add up to $n$. 

Furthermore, this theorem captures the intuition that not all noises are created equal, in the sense that noise at key locations should have a  higher contribution to the limiting disagreement. Indeed, in the event that noises at different nodes are uncorrelated, the second term of Theorem \ref{mainthm} is easily seen to be zero (since the matrix $H_{P^2}$ has zero diagonal by definition) and we obtain \begin{equation} \label{uceq} \delta_{\rm ss}\left(P, {\rm diag} \left(\sigma_1^2, \ldots, \sigma_n^2 \right) \right) = \sum_{i=1}^n \sum_{j=1}^n \sigma_i^2 \pi_i^2 \pi_j H_{P^2} (j \rightarrow i). \end{equation} We see that in this case $\delta_{\rm ss}$ is a linear combination of the variances at each node, where the variance $\sigma_i^2$ multiplied by $\pi_i^2 \sum_{j=1}^n \pi_j H_{P^2}(j \rightarrow i)$. Note that this multiplier  is a product of a measure of importance coming from the stationary distribution (i.e., $\pi_i^2$) and a measure of the ``mean accessibility'' of a node (i.e., $\sum_{j=1}^n \pi_j H_{P^2}(j \rightarrow i)$).


In the event that all noises have the same variance, we obtain the simplified version 
\begin{equation} \label{samevar} \delta_{\rm ss}\left(P, \sigma^2 I \right)  = \sigma^2  \sum_{i=1}^n \sum_{j=1}^n  \pi_i^2 \pi_j H_{P^2} (j \rightarrow i). \end{equation} As we discuss later in this paper, for many classes of matrices $P$ the quantity $\sum_{i=1}^n \sum_{j=1}^n  \pi_i^2 \pi_j H_{P^2} (j \rightarrow i)$ grows with the 
total dimension of the system $n$. In other words, although the system is technically stable in the sense of having bounded expected disagreement as $t \rightarrow \infty$, this stability is almost meaningless if $n$ is large. Equations (\ref{uceq}) and (\ref{samevar}) allow us to determine when this is the case by analyzing how stationary distribution and hitting times grow on various kinds of graphs\footnote{In particular, an implication is that the the amount of noise amplification in the network $\delta_{\rm ss}(P, \sigma^2 I)$ is not fully characterized by the spectral gap of the underlying graph; see, for example, the table in Section \ref{examples}.}. Later in the paper (in Section \ref{examples}) we will use these equations to work out how $\delta_{\rm ss}$ scales for a variety of matrices $P$ which come from graphs.

\subsection{Proof of Theorem \ref{mainthm}}

We now turn to the proof of Theorem \ref{mainthm}. We begin with an informal sketch of the main idea. 

First, the disagreement $\delta_{\rm ss}^{\rm uni}$ may be thought of as proportional to the trace of a certain asymptotic covariance matrix $\Sigma_{\rm ss}$ (to be formally defined later), whereas the weighted disagreement $\delta_{\rm ss}$ may be thought of as the trace of $D_{\rm \pi} \Sigma_{\rm ss}$ (since multiplication by the diagonal matrix $D_{\rm \pi}$ scales the $i$'th diagonal entry of covariance matrix $\Sigma_{\rm ss}$ by $\pi_i$).  Now the key observation is that we can write down a matrix $\Sigmah$ with the properties that (i) the trace of $\Sigmah$ is the same as the trace of $D_{\rm \pi} \Sigma_{\rm ss}$ (ii) the difference between the matrix $\Sigmah$ and the matrix $H D_{\rm \pi} \Sigma_{\rm w} D_{\rm \pi}$ is a matrix with constant columns. These two properties allow us to relate $\delta_{\rm ss}$ to the matrix $H D_{\rm \pi} \Sigma_{\rm w} D_{\rm \pi}$ and very quickly lead to the proof of Theorem \ref{mainthm}. 

The existence of such a matrix $\Sigmah$ is a technical observation and we are not aware of any intuitive explanation or justification for it. As a consequence, the proof given next is not very intuitive and largely composed of the manipulation of equations.

  Specifically, we begin by deriving a recursion for the error covariance matrix at time $t$ and show that, after a large number of equation manipulations, in the limit as $t \rightarrow \infty$ it leads to a certain representation of $D_{\rm \pi} \Sigma_{\rm ss}$ as an infinite sum; we then rearrange some parts of this sum to define the matrix $\Sigmah$; finally, we show that the matrix $\Sigmah$ thus defined has the properties (i) and (ii) above and immediately deduce Theorem \ref{mainthm}. 

We now begin the proof itself. The matrix $J$ defined as $$J := \1 \pi^T,$$ will be of central importance to the proof; the following lemma collects a number of its useful properties. 

\begin{lemma}[Properties of the Matrix $J$] \label{jprop} \begin{eqnarray*} \x(t) & = & J x(t), \\ 
J \1 & = & \1, \\ 
J P & = & J, \\
P J & = & J, \\
J^2 & = & J, \\
(I-J)^2 & = & I - J, \\
(P^l-J)^k & = & P^{lk} - J, \mbox{          } l = 0,1,2, \ldots,  \mbox{  and     } k = 1,2,\ldots\\
\rho(P-J) & < & 1.
\end{eqnarray*} 
\end{lemma}

\begin{proof} The first six equations are immediate consequences of the definitions of $J$, $P$, and $\pi$.  The seventh equation can be established by induction. Indeed, the base case of $k=1$ is trivial. If the identity is established for some $k$, then  
\begin{eqnarray*} (P^l-J)^{k+1} & = &  (P^l-J)(P^l-J)^k \\ & = &  (P^l-J)(P^{lk} - J) \\  & = & P^{l(k+1)} - P^l J - JP^{lk} + J^2 \\ & = &  P^{l(k+1)} - J. 
\end{eqnarray*} Note that some care is needed in applying the seventh equation as it is obviously false when $k=0$.

To prove the final inequality suppose that for some vector $v \in \mathbb{C}^n$ and some $\lambda \in \mathbb{C}$, $$(P-J) v = \lambda v.$$ If $\lambda \neq 0$, then \begin{small} \[ \pi^T v =  \pi^T P v =  \pi^T (P-J) v + \pi^T J v = \lambda \pi^T v + \pi^T v = (1 + \lambda) \pi^T v \] \end{small} which implies that $\pi^T v = 0$. In turn, this implies that $Jv=0$ and consequently $v$ is an eigenvector of $P$ with eigenvalue $\lambda$. By stochasticity of $P$, this  implies $ |\lambda| \leq 1$. 

To show the strict inequality, observe that  since the matrix $P$ is irreducible and aperiodic, we have that it has only one eigenvector with an eigenvalue that has absolute value $1$ and that is the all-ones vector $\1$. Thus if $|\lambda| = 1$ then the vector $v$ is a multiple of $\1$; however, ${\bf 1}$ is an eigenvalue of $P-J$ with eigenvalue zero which contradicts $|\lambda|=1$. We conclude  that if $\lambda \neq 0$ then $|\lambda| < 1$, which is what we needed to show. 
\end{proof}

\medskip

Next, we define the matrix  
$$\Sigma(t) := E[ e(t) e(t)^T].$$ 

The following lemma derives a recursion satisfied by $\Sigma(t)$. 

\begin{lemma}[Simplified Recursion for the Covariance Matrix] \label{srecur} \begin{eqnarray}  \Sigma(t+1) & = & (P-J) \Sigma(t) (P-J)^T +  (I - J) \Sigma_{w} (I-J)^T. \nonumber 
\label{onestepupdate}
\end{eqnarray} 
\end{lemma} 

\begin{proof} Indeed, using Lemma \ref{jprop},
\begin{eqnarray*} e(t+1) & = & x(t+1) - J x(t+1) \\ 
& = & P x(t) + w(t) - JP x(t) - J w(t) \\ 
& = & (P-J) x(t) + (I-J) w(t) \\ 
& = & (P-J) (x(t) - \overline{x}(t)) + (I-J) w(t) \\
& = & (P-J) e(t) + (I-J) w(t), 
\end{eqnarray*} and therefore,  \begin{small}
\begin{eqnarray*} \Sigma(t+1) &  = & E[~e(t+1) e(t+1)^T] \\ & = & E \left[ \left( (P-J) e(t)  \right. \right. \\ && \left. \left. + (I-J) w(t) \right)  \left( e(t)^T (P-J)^T + w(t)^T (I-J)^T \right) ~\right], 
\end{eqnarray*} \end{small} and finally since $E[e(t) w(t)^T] = E[w(t) e(t)^T] =0$, this immediately implies the current lemma. 
\end{proof}

\smallskip

Observe that \[ \delta_{\rm ss} = \lim \sup_{t \rightarrow \infty} \sum_{i=1}^n \pi_i [\Sigma(t)]_{ii}. \] As a consequence of Lemma \ref{srecur}, it is not hard to see that the initial condition $x(0)$ has no influence on $\delta_{\rm ss}$. Indeed, using $\Sigma^0(t)$ to denote what $\Sigma(t)$ would be if $x(0)=0$ we  have that  
\[ \Sigma(t) = \Sigma_0(t) + (P-J)^t e(0)  e(0)^T \left( (P-J)^T \right)^{t}. \] Since $\rho(P-J)<1$ by Lemma \ref{jprop},  we see that $\Sigma(t) - \Sigma^0(t) \rightarrow 0$. Using $\delta_{\rm ss}^0$ to denote what $\delta_{\rm ss}$ would be if $x(0)=0$, we have that 
\begin{eqnarray*} \delta_{\rm ss} - \delta_{\rm ss}^0  &= &   \lim \sup_{t \rightarrow \infty} \left( \pi_i [\Sigma^0(t)]_{ii} + \pi_i [\Sigma(t) - \Sigma^{0}(t)]_{ii} \right) \\ && - \lim \sup_{t \rightarrow \infty}  \pi_i [\Sigma^0(t)]_{ii}  \\ & = &  0.
\end{eqnarray*}

{\em Thus for the remainder of this paper, we will make the assumption that $x(0)=0$, i.e., that the initial condition is the origin.} This assumption will slightly simplify some of the expressions which follow. 

In our next corollary, we write down an explicit expression for $\Sigma(t)$ as an infinite sum.

\begin{corollary}[Explicit Expression for the Covariance Matrix]  For $t \geq 1$, \[ \Sigma(t) =\sum_{k=0}^{t-1} (P^k - J) \Sigma_{ w} ((P^T)^k - J^T). \]
\end{corollary} 

\begin{proof} Indeed, as we are now assuming that $x(0)=0$, Lemma \ref{srecur} implies that for $t \geq 1$,  \begin{small}
\begin{align} \Sigma(t) & =   \sum_{k=0}^{t-1} (P-J)^k (I-J) \Sigma_{w} (I-J)^T (P^T - J^T)^k \nonumber \\ 
& =   (I-J)D\Sigma_{w} (I-J^T) \nonumber \\ & +  \sum_{k=1}^{t-1} (P^k - J) (I - J)\Sigma_{w} (I- J)^T ((P^T)^k - J^T) \label{sigmaeq}
\end{align} \end{small} where the last line used Lemma \ref{jprop} for the equality $(P-J)^k = P^k - J$ when $k \geq 1$.

Next, observing that by
Lemma \ref{jprop}, again if $k \geq 1$, \[ (P^k - J) J = (P-J)^k J  = (P-J)^{k-1} (P-J) J =  0\] and therefore if $k \geq 1$, \begin{footnotesize}
\[ (P^k - J)(I-J)  \Sigma_{w}  (I-J)^T ((P^T)^k-J^T) = (P^k - J)  \Sigma_{w}  ((P^T)^k - J^T). \]   \end{footnotesize} Plugging this into Eq. (\ref{sigmaeq}),  we obtain the statement of the corollary.
\end{proof} 

\medskip

Appealing once again to Lemma \ref{jprop}, we may rewrite the previous corollary as 
\[ \Sigma(t) = (I-J) \Sigma_{\rm w}(I-J)^T +  \sum_{k=1}^{t-1} (P-J)^k \Sigma_{\rm w} ((P-J)^T)^k.  \]
Furthermore, by Lemma \ref{jprop} the matrix  $P-J$ has spectral radius strictly less than $1$. It follows that we can define  
\begin{equation} \label{firstsig}  \Sigma_{\rm ss} :=  (I - J) \Sigma_{\rm w}(I-J)^T +  \sum_{k=1}^{\infty} (P-J)^k \Sigma_{\rm w}  ((P-J)^T)^k, \end{equation} and this is a valid definition since the  
the sum on the right-hand side converges. Moreover,  \[ \Sigma_{\rm ss} = \lim_{t \rightarrow \infty} \Sigma(t). \]

Our next step is to observe that if we define $D_{\pi} := {\rm diag}(\pi_1, \pi_2, \ldots, \pi_n)$, then the quantity $\delta_{\rm ss}$ we are seeking to characterize can be written as 
\begin{equation} \label{treq} \delta_{\rm ss} = {\rm Tr} (\Sigma_{\rm ss} D_{\pi} ). \end{equation}

We therefore now turn our attention to the matrix $\Sigma_{\rm ss} D_{\pi}$. Our next lemma derives an explicit expression for this matrix as an infinite sum.  The proof of this lemma is the 
only place in the proof of Theorem \ref{mainthm} where we use the reversibility of the matrix $P$.

\begin{lemma}[Explicit Expression for the Weighted Covariance Matrix] \begin{small}  \[ \Sigma_{\rm ss} D_{\pi} =  (I-J) \Sigma_{\rm w}D_{\pi} (I-J) + \sum_{k=1}^{\infty} (P-J)^k \Sigma_{\rm w} D_{\pi} (P - J)^k. \] \end{small} \label{sdlemma}
\end{lemma} 
\begin{proof} Indeed, from Eq. (\ref{firstsig}), 
\begin{equation} \Sigma_{\rm ss} D_{\pi} =  (I-J) \Sigma_{\rm w} (I-J)^T D_{\pi} +  \sum_{k=1}^{\infty} (P-J)^k \Sigma_{\rm w} (P^T - J^T)^k D_{\pi} \label{firstsd} \end{equation} Now the reversibility of $P$ means that for all $i,j = 1, \ldots, n$, we have that $\pi_i P_{ij} = \pi_j P_{ji}$. We can write this in matrix form as  
\[ D_{\pi} P = P^T D_{\pi}. \] One can also verify directly from the definitions of $J$ and $D_{\pi}$ that 
\[ D_{\pi} J = J^T D_{\pi}. \] Plugging the last two equations into Eq. (\ref{firstsd}), we obtain the statement of the lemma. 
%
\end{proof} 

\medskip

{\bf Definition of the matrix $\Sigmah$:} we would now like to introduce the matrix  $\Sigmah$ defined as 
\begin{equation} \label{shdef} \Sigmah :=  \sum_{k=0}^{\infty} (P^{2k}-J) \Sigma_{\rm w} D_{\pi}. \end{equation} As before, by Lemma \ref{jprop} we have that $\rho(P-J)<1$, and consequently the sum on the right hand side converges and $\Sigmah$ is well defined. Furthermore, since ${\rm Tr}(AB) = {\rm Tr}(BA)$, Lemma \ref{sdlemma} immediately implies that \[ {\rm Tr}(\Sigmah) = {\rm Tr}(\Sigma_{\rm ss} D_{\pi}), \] and putting this together with Eq. (\ref{treq}), we have \begin{equation} \label{shtr} {\rm Tr}(\Sigmah) = \delta_{\rm ss}. \end{equation} Furthermore, since by Lemma \ref{jprop} we have that  $J(P^k-J)=0$ for all $k \geq 0$, we have that 
\begin{equation} \label{jorth} J \Sigmah = 0. \end{equation} Finally,  using Eq. (\ref{jorth}), followed by Eq. (\ref{shdef}) and Lemma \ref{jprop}, we have the following sequence of equations:
\begin{small} \begin{eqnarray*} P^2 \Sigmah & = & (P^2-J) \Sigmah \\ 
& = & \sum_{k=0}^{\infty} (P^2 - J) (P^{2k}-J) \Sigma_{\rm w}D_{\pi} \\ 
& = & (P^2 - J) (I-J) \Sigma_{\rm w} D_{\pi} +  \sum_{k=1}^{\infty} (P^2 - J) (P^2 - J)^k  \Sigma_{\rm w} D_{\pi} \\
& = &  (P^2 - J) \Sigma_{\rm w} D_{\pi} + \sum_{k=1}^{\infty} (P^{2(k+1)} - J) \Sigma_{\rm w} D_{\pi} \\
& = &  \sum_{k=0}^{\infty} (P^{2(k+1)} - J) \Sigma_{\rm w} D_{\pi} \\ 
& = &  \sum_{k=1}^{\infty} (P^{2k} - J) \Sigma_{\rm w} D_{\pi} \\ 
& = & \Sigmah - (I-J) \Sigma_{\rm w} D_{\pi} 
\end{eqnarray*} \end{small} which we may rearrange as 
\begin{equation} \label{sheq} \Sigmah = P^2 \Sigmah +  (I-J) \Sigma_{\rm w} D_{\pi} 
\end{equation}

\bigskip

With these identities in place, we are finally ready to prove Theorem \ref{mainthm}. 

\medskip

\begin{proof}[Proof of Theorem \ref{mainthm}] Let us stack up the hitting times in the Markov chain which moves according to $P^2$ in the matrix $H$, i.e., $H_{ij} := H_{P^2}(i \rightarrow j)$. By conditioning on what happens after a single step, we have the usual identity    
\[ H_{ij} = 1 + \sum_{k=1}^n [P^2]_{ik} H_{kj} , ~~~~~~~~~ i \neq j. \] On the other hand, since a random walk spends an expected $1/\pi_i$ steps in between visits to node $i$, \[ H_{ii} = 0 = 1 + \sum_{i=1}^n [P^2]_{ik} H_{ki} - \frac{1}{\pi_i}. \] We can the previous two equations in matrix form together as
\[ H = \1 \1^T + P^2 H - D_{\pi}^{-1} , \] or
\[  (I-P^2) H = \1 \1^T - D_{\rm \pi}^{-1}.   \] Multiplying both sides of this equation by $D_{\rm \pi} \Sigma_{\rm w} D_{\rm \pi} $ on the right, we  obtain   \begin{equation} \label{heq} (I - P^2) H D_{\rm \pi}^2 D_{\sigma^2} = (J - I) \Sigma_{\rm w} D_{\pi} . \end{equation} On the other hand, observe that we may rearrange Eq. (\ref{sheq}) as 
\begin{equation} \label{seq}  (I - P^2) \Sigmah =   (I - J) \Sigma_{\rm w}D_{\pi} . \end{equation}

Adding Eq. (\ref{heq}) and Eq. (\ref{seq}), we obtain \[ \left( I - P^2 \right) \left( \Sigmah + H D_{\rm \pi} \Sigma_{\rm w} D_{\rm \pi} \right) = 0,\] meaning that all the columns $ \Sigmah + H D_{\rm \pi} \Sigma_{\rm w} D_{\rm \pi}$ lie in the null space of $I - P^2$. But because $P$ is irreducible and aperiodic, the null space of $I-P^2$ is ${\rm span} \{\1\}$. Thus $\Sigmah + H D_{\rm \pi} \Sigma_{\rm w} D_{\rm \pi}$ is a matrix with constant 
columns. In other words, there exists a vector $v$ such that 
\begin{equation} \label{solved} \Sigmah= -  H D_{\rm \pi} \Sigma_{\rm w} D_{\rm \pi} + \1 v^T. \end{equation} 

We can, in fact, compute $\1 v^T $ exactly by utilizing Eq. (\ref{jorth}), which implies that 
\[ \1 \pi^T H D_{\pi}^2  D_{\sigma^2}= \1 v^T. \] Plugging this this into Eq. (\ref{solved}), we obtain  
\begin{equation} \label{sigmaheq} \widehat{\Sigma} =  - H  D_{\rm \pi} \Sigma_{\rm w} D_{\rm \pi}+ \1 \pi^T  H D_{\rm \pi} \Sigma_{\rm w} D_{\rm \pi}.  \end{equation}

Finally recalling that $\delta_{\rm ss}$ is the trace of $\Sigmah$ (see Eq. (\ref{shtr})),  
\[ \delta_{\rm ss} = - {\rm Tr}  (H  D_{\rm \pi} \Sigma_{\rm w} D_{\rm \pi}) + \pi^T H D_{\rm \pi} \Sigma_w \pi. \]
\end{proof}

\smallskip

Having proven  Theorem \ref{mainthm}, we conclude the section with a discussion of its simplifications in the case when $P$ is symmetric, followed by an enumeration of some
connections it implies between the weighted steady-state disagreement $\delta_{\rm ss}$, the unweighted steady-state disagreement $\delta_{\rm ss}^{\rm uni}$, and
other graph-theoretic quantities such as the electrical resistance and the Kemeny constant. 

\subsection{Simplifications of Theorem \ref{mainthm} in the symmetric case\label{sym.1}} In this subsection we collect several simplifications and observations that pertain to symmetric matrices $P$. {\em Thus for the remainder of this Section \ref{sym.1}, we will asume that $P$ is a symmetric matrix}. Some of the identities we derive in this brief subsection will be new, whereas others will be simple proofs of already known results. The main reason these results are collected here is that we will need to use some of them in Sections \ref{examples} and \ref{symmetric}.

 Since the symmetry of $P$ implies that $\pi = (1/n) \1$, we immediately obtain that \begin{equation} \label{symmdelta} \delta_{\rm ss}(P, \Sigma_w) = \frac{1}{n^3} \1^T H \Sigma_w \1 - \frac{1}{n^2} {\rm Tr} (H \Sigma_w).  \end{equation} Using the notation $\Sigma_w = [\sigma_{ij}]$ as well as the fact that $\Sigma_w$ is symmetric, we may expand this expression to obtain 
\begin{align} \label{expand} \delta_{\rm ss}(P, \Sigma_w) & =  \frac{1}{n^3} \sum_{i=1}^n \sum_{k=1}^n \sum_{l=1}^n H_{P^2}(k \rightarrow l)  \sigma_{li} \nonumber \\ & - \frac{1}{n^2} \sum_{i<j} \sigma_{ij} \left( H_{P^2}(i \rightarrow j) + H_{P^2} (j \rightarrow i) \right) \end{align} It is also worthwhile to rewrite this as 
\begin{eqnarray} \delta_{\rm ss}(P, \Sigma_w)  & = & \frac{1}{n^3} \left( {\rm Tr}( H \Sigma_w \1 \1^T) - {\rm Tr}(n H \Sigma_w) \right) \nonumber \\
& = & -\frac{1}{n^3} {\rm Tr} ( H \Sigma_w (nI - \1 \1^T ) ) \nonumber \\
& = & -\frac{1}{n^2} {\rm Tr} (H \Sigma_w  (I - (1/n) \1 \1^T)) \nonumber \\
& = & -\frac{1}{n^2} {\rm Tr} (H \Sigma_w P_{\1^\perp} ) \label{deltaperp}
\end{eqnarray} where $P_{\1^\perp} = I  - (1/n) \1 \1^T$ is the orthogonal projection matrix onto the subspace $\1^\perp$.

Equations (\ref{symmdelta}), (\ref{expand}), and (\ref{deltaperp}) are considerable simplifications of Theorem \ref{mainthm}. However, it is possible to  simplify Theorem \ref{mainthm} still further if we additionally assume that $\Sigma_w$ is diagonal, i.e., $\Sigma_w = {\rm Diag}(\sigma_1^2, \ldots, \sigma_n^2)$. In that case, the second term on the right of Eq. (\ref{symmdelta}) is zero and we obtain \begin{equation} \label{symdiag} \delta_{\rm ss}(P, {\rm diag}(\sigma_1^2, \ldots, \sigma_n^2)) = \frac{1}{n} \sum_{i=1}^n \sum_{k=1}^n \frac{\sigma_i^2 H_{P^2}(k \rightarrow i)}{n^2}. \end{equation} 

Finally, let us assume that the the variances are all identical, i.e., $\Sigma_w = \sigma^2 I$. In this case the answer can be written in a particularly simple form in terms of the so-called Kemeny constant.

\smallskip

\noindent {\bf Kemeny constant.} A classic result of Kemeny sometimes called the ``random target lemma'' shows that the quantity $\sum_{j=1}^n \pi_j H_M(i \rightarrow j)$ is independent of $i$ for any Markov chain $M$. The quantity $\sum_{j=1}^n \pi_j H_M(i \rightarrow j)$  is thus called the Kemeny constant of the Markov chain and we will denote it by $K(M)$. 

With this in mind, from Eq. (\ref{symdiag}) we have that
\begin{equation} \label{kemeny}  \delta_{\rm ss}(P, \sigma^2 I) = \sigma^2 \frac{K(P^2)}{n}. 
\end{equation}

Arguably, this is the simplest possible characterization of $\delta_{\rm ss}$ for symmetric matrices $P$ and $\Sigma_w = \sigma^2 I$. 

Moreover, we remark that this can be rewritten in terms of the eigenvalues of the matrix $P$. Indeed, defining $\Lambda(M)$ to be the set of all non-principal eigenvalues of $M$, it is known \cite{kem2, kirkland}   that  
\begin{equation} \label{keig} K(M) = \sum_{\lambda \in \Lambda(M)} \frac{1}{1 - \lambda}. \end{equation} Putting the last two equations together, we have that for symmetric $P$ with constant variances,
\[ \delta_{\rm ss}(P, \sigma^2 I) = \frac{\sigma^2}{n} \sum_{\lambda \in \Lambda(P)} \frac{1}{1 - \lambda^2}. \] This last identity is not a new result; rather, it was first observed in \cite{boyd} where it was  proved directly by diagonalizing $P$. 

\medskip

\noindent {\bf Electrical resistance.} We remark that it is possible to use Theorem \ref{mainthm} obtain a characterization of $\delta_{\rm ss}(P, \sigma^2 I)$ in terms of electric resistances as first shown in \cite{zmp} (see also \cite{bamieh2} for the analogous observation in continuous time). 

Given a reversible stochastic matrix $M \in \R^{n \times n}$ with zero diagonal, we define \[ q_M(x,y) := \pi_x M(x,y). \] Note that reversibility of $M$ implies that $q_M(x,y) = q_M(y,x)$. The quantity  $R_M ( a \leftrightarrow b)$ is defined to be the resistance from $a$ to $b$ in the electrical network where the edge $(i,j)$ is replaced with a resistor with resistance $1/q_M(i,j)$. 

There is a connection between resistances, thus defined, and hitting times:
\begin{equation} \label{hsum} H_M(i \rightarrow j) + H_M(j \rightarrow i) = R_{M} (i \leftrightarrow j).\end{equation} A proof may be found in Section 10.3 of \cite{lpw}.

Using this identity along with the symmetry of the matrix $P$ (which implies all $\pi_i$ equal $1/n))$,  we can group terms together in Eq. (\ref{symdiag}) to obtain 
\[ \delta(P, \sigma^2 I ) =  \frac{\sigma^2}{n} \frac{\sum_{i<j} R_{P^2}(i \leftrightarrow j)}{n^2}. \] As mentioned above, this identity was first proved in \cite{zmp}.

\subsection{Further connections to resistance, the Kemeny constant, and unweighted steady-state disagreement.}  We now turn our attention back to the case when $P$ is reversible (and not necessarily symmetric). In this subsection, we derive a number of inequalities bounding $\delta_{\rm ss}$ in terms of the largest resistance and the Kemeny constant. We also discuss how we can bound $\delta_{\rm ss}^{\rm uni}$ in terms of $\delta_{\rm ss}$. All the inequalities derived within this subsection are new.

By putting Theorem \ref{mainthm} together with Eq. (\ref{hsum}), we obtain
\begin{small} \begin{align*} \delta_{\rm ss}(P, {\rm diag}(\sigma_1^2, \ldots, \sigma_n^2)) & =  \sum_{i=1}^n \sum_{j=1}^n \sigma_i^2 \pi_i^2 \pi_j H_{P^2}(i \rightarrow j) \\ 
& =  \left( \max_{i=1, \ldots, n} \sigma_i^2 \pi_i  \right) \left( \max_{i,j} R_{P^2}(i \leftrightarrow j) \right).
\end{align*}   \end{small}

In other words, $\delta_{\rm ss}$ may be quickly bounded in terms of the largest variance, stationary distribution, and resistance. We can also obtain a  lower bound in terms 
of the smallest versions of similar quantities. Indeed: 

\begin{small}
\begin{align*} \delta_{\rm ss}(P, {\rm diag}(\sigma_1^2, \ldots, \sigma_n^2)) & =  \sum_{i=1}^n \sum_{j=1}^n \sigma_i^2 \pi_i^2 \pi_j H_{P^2}(i \rightarrow j) \\ 
& \geq  \left( \min_{i=1, \ldots, n} \sigma_i^2 \pi_i  \right)  \sum_{i=1}^n \sum_{j=1}^n \pi_i \pi_j H_{P^2}(j \rightarrow i) \\
& =  \left( \min_{i=1, \ldots, n} \sigma_i^2 \pi_i  \right)  K(P^2).
\end{align*} 
\end{small}  These inequalities can be used to obtain quick bounds on $\delta_{\rm ss}$ when 
either the resistance of the Kemeny constant are known.

\bigskip

\noindent {\bf Bounding $\delta_{\rm ss}^{\rm uni}$.} The problem of giving a combinatorial characterization of $\delta_{\rm ss}^{\rm uni}(P, \Sigma_w)$ for reversible $P$ is open, to the best of our knowledge. Here we provide
combinatorial lower and upper bounds on $\delta_{\rm ss}^{\rm uni}$ which are  tighter than the best previously known bounds.

Indeed, observe that \[ \delta_{\rm ss}(t) = \sum_{i=1}^n \pi_i E[e_i^2(t)] = \frac{1}{n} \sum_{i=1}^n n\pi_i E[e_i^2(t)], \] so that 
\[ n \pi_{\rm min} \delta_{\rm ss}^{\rm uni}(t) \leq  \delta_{\rm ss}(t) \leq n \pi_{\rm max} \delta_{\rm ss}^{\rm uni}(t),\]  which implies
\[ \frac{\delta_{\rm ss}}{n \pi_{\rm max}} \leq \delta_{\rm ss}^{\rm uni} \leq \frac{\delta_{\rm ss}}{n \pi_{\rm min}} \] 
Thus as a consequence Eq. (\ref{uceq}),  we have \begin{small}
 \begin{align*} \frac{1}{n \pi_{\rm max}}  \sum_{i=1}^n \sum_{j=1}^n \sigma_i^2 \pi_i^2 \pi_j H_{P^2}(j \rightarrow i) & \leq \delta_{\rm ss}^{\rm uni}(P, {\rm diag}(\sigma_1^2, \ldots, \sigma_n^2) ) \end{align*} and 
\begin{align*}  \frac{1}{n \pi_{\rm min}} \sum_{i=1}^n \sum_{j=1}^n \sigma_i^2 \pi_i^2 \pi_j H_{P^2}(j \rightarrow i) & \geq \delta_{\rm ss}^{\rm uni}(P, {\rm diag}(\sigma_1^2, \ldots, \sigma_n^2) ) . \end{align*} \end{small} This pair of bounds may be viewed as an improvement on the results of \cite{zmp}. That paper provided upper and lower bounds on $\delta_{\rm ss}^{\rm uni}$ in terms of the stationary distribution and the electrical resistance; the ratio of the upper and lower bounds given was $(\pi_{\rm max}/\pi_{\rm min})^4$. By contrast, the ratio of the upper and lower bounds in the two equations above is $\pi_{\rm max}/\pi_{\rm min}$.

%

\section{Examples\label{examples}} 

The goal of this section is to demonstrate that ``back of the envelope'' calculations based on Theorem \ref{mainthm} can often be used to give order-optimal estimates of $\delta_{\rm ss}$. Indeed, we will obtain estimates of how $\delta_{\rm ss}$ scales with the number of nodes on many common graphs. The interested reader may skip ahead to the table at the end of this section.

 We begin by describing a natural way in which a stochastic matrix  can be chosen from a graph. Given an {\em undirected} connected graph $G=(\{1, \ldots, n\},E)$ without self-loops, let $d(i)$ denote the degree of node $i$, and let us define  
\begin{equation} \label{pdef} \widetilde{P}_{ij}  = \begin{cases} 1/d(i) & (i,j) \in E, \\ 0  & \mbox{ else. } \end{cases} 
\end{equation} Clearly, $\widetilde{P}$ is a stochastic matrix. However, if the graph $G$ is bipartite the quantity $\delta_{\rm ss}(\widetilde{P}, {\rm diag}(\sigma_1^2, \ldots, \sigma_n^2))$ will be infinite if at least one of $\sigma_i^2$ is strictly positive\footnote{We relegate the justification of this assertion to a footnote. Indeed, suppose that the graph $G$ is bipartite and let $V_1 \cup V_2 = \{1, \ldots, n\}$ be a bipartition. Then the vector $v$ defined as $v_i = d(i), i \in V_1$ and $v_i = -d(i), i \in V_2$ is a left-eigenvector of $P$ with eigenvalue $-1$. Observe that $v^T \1 = 0$ since both $\sum_{i \in V_1} d(i)$ and $\sum_{i \in V_2} d(i)$ count the number of edges going between $V_1$ and $V_2$. Thus 
\[ v^T e(t+1) = v^T x(t+1) = - v^T x(t) + v^T w(t) = - v^T e(t) + v^T w(t). \] Letting $y(t) = (-1)^t v^T e(t)$ this becomes $$y(t+1) = y(t) + (-1)^{t+1} v^T w(t).$$ Since $x(0)=0$ we have $E[v^T e(t)]=0$ and $E[y(t)]=0$. Thus as long as at least one $\sigma_i^2$ is strictly positive, we have that ${\rm Var}(y(t)) \rightarrow +\infty$ and consequently ${\rm Var} (v^T e(t)) \rightarrow +\infty$. Since all $\pi_i$ are strictly positive due to the connectivity of $G$, it is not too hard to see that this implies that $\delta_{\rm ss}$ is infinite.}. An easy fix for this is to consider instead  \begin{equation} P = \frac{1}{2} I + \frac{1}{2} \widetilde{P}. \label{pdef2} \end{equation}  Intuitively, each agent will place half of its weight on itself and distribute half uniformly among neighboring agents. 
It is tautological that if $G$ is connected then $P$ is irreducible. Finally, observe that $P$ constructed this way is  always reversible.

After attending to some preliminary remarks in the next subsection, we proceed to give order-optimal estimates of the quantity $\delta_{\rm ss}(P, {\rm diag}(\sigma_1^2, \ldots, \sigma_n^2))$ for a number of matrices $P$ constructed from graphs in this way.

\medskip

\noindent {\bf Preliminary observations.} 

\medskip

\begin{itemize} 
\item We  note that it is quite easy to compute the stationary distribution of a matrix defined from an undirected graph according to Eq. (\ref{pdef}, \ref{pdef2}). Indeed, letting $m$ be the number of edges in the graph $G$ which are not self loops, it is easy to verify that $\pi_i = d(i)/(2m)$. Naturally, this is also the stationary distribution of $P^2$ and $\widetilde{P}$.

\medskip

\item We remind the reader that for two functions $f, g: X \rightarrow \R$, the notation $f(x) = \Theta( g(x) )$ means that there exist positive  numbers $c,C$ such that $c g(x) \leq f(x) \leq C g(x).$  We will sometimes write this as 
$f(x) \simeq g(x).$

\medskip

\item Observe that on conntected graphs where the total number of edges is linear in $n$, we have that $\pi_i \simeq 1/n$ for all $i$ and consequently $\delta_{\rm ss}^{\rm uni}  \simeq \delta_{\rm ss}$.

\medskip

\item Let us adopt the notation ${\cal H}_M$ for the largest hitting time in the chain which moves according to the stochastic matrix $M$, i.e., ${\cal H}_{M} = \max_{i,j} H_M(i \rightarrow j)$. Then we have the following lemma. 

\begin{lemma} If $M$ is diagonally dominant, then
\[  {\cal H}_{M^2} = O \left(   {\cal H}_M \right). \]  \label{sqrlemma}
\end{lemma} Although this statement is elementary, we provide a proof for completeness. 
\begin{proof}  Consider any pair of nodes $i,j$ and let $T_M(i \rightarrow j)$ be the first time a random walk starting at $i$ and moving according to $M$ hits $j$, i.e., $T_M(i \rightarrow j)$ is the
random variable whose expectation is $H_M(i \rightarrow j)$.  Then, as a consequence of the diagonal dominance of $M$, we have that for any time $t$,
\begin{equation} \label{squareineq} P(T_{M^2}(i \rightarrow j) \leq  t) \geq \frac{P(T_M(i \rightarrow j) \leq 2t)}{2}. \end{equation} This is true because: \begin{itemize} \item The probability of the event $\{T_{M^2}( i \rightarrow j) \leq t\}$ equals the probability that a random walk starting at $i$ and moving according to $M$ hits $j$ by time $2t$ at an even time step. \item Consider a sample path in the chain moving according to $M$ which starts at $i$ and ends when it hits $j$, which happens by time $2t$. Either (a) this sample path hits $j$ at an even time step (b) this sample path can be extended with a self-loop to further hit $j$ at an even time step by time $2t$ (and the probability of taking that self-loop is at least $1/2$ by diagonal dominance). 
\end{itemize}

We next plug $t = {\cal H}_M$ into Eq. (\ref{squareineq}) to obtain 
\[ P(T_{M^2}(i \rightarrow j) \leq {\cal H}_M) \geq \frac{P(T_M(i \rightarrow j) \leq 2 {\cal H}_M)}{2} \geq \frac{1}{4}, \] where the last step used Markov's inequality. Since this did not depend on the starting point $i$, we can iterate this argument to obtain that $E[T_{M^2}(i \rightarrow j)] \leq 4 {\cal H}_M$, which is what we needed to show. 
\end{proof} 

\medskip

\item As a consequence of the last bullet as well as the fact that $H_{P}(i \rightarrow j) = 2 H_{\widetilde{P}}(i \rightarrow j)$ is that ${\cal H}_{P^2} \simeq {\cal H}_{\widetilde{P}}$. 

\medskip

\item A convenient tool to compute upper bounds on hitting times in $\widetilde{P}$ is their connection to electric resistances. We refer the reader back to Section \ref{sym.1} for the definition of electric resistance $R_M(i \leftrightarrow j)$  and here merely recall the identity 
\begin{equation} \label{hsum} H_M(i \rightarrow j) + H_M(j \rightarrow i) = R_{M} (i \leftrightarrow j)\end{equation} 

For the matrix $\widetilde{P}$ defined in Eq. (\ref{pdef2}) we have that for every pair of neighbors $x,y$, $$q_P(x,y) = \frac{1}{d(x)} \frac{d(x)}{2m} = \frac{1}{2m},$$ where recall $m$ is the number of edges in the graph $G$. {\em Consequently, the resistance $R_{\widetilde{P}}(i \leftrightarrow j)$ can be obtained as electrical resistance between $i$ and $j$ in a graph where 
every edge has resistance $2m$.}

\smallskip

%
%

\end{itemize}

With these preliminary remarks in place, we now turn to the problem of computing $\delta_{\rm ss}$ for matrices which come from graphs according to Eq. (\ref{pdef}, \ref{pdef2}).  {\em We will be assuming that  $\Sigma_w = {\rm diag}(\sigma_1^2, \ldots, \sigma_n^2)$ for the remainder of this section (and in places we will even consider the case when all $\sigma_i^2$ are equal to the same $\sigma^2$)}.  As we will see next, we can use Theorem \ref{mainthm} as well as the above preliminary observations to estimate $\delta_{\rm ss}$ to within a constant multiplicative factor for a number of common graphs. 

\medskip

\noindent {\bf The complete graph.} By symmetry $\pi_i = 1/n$ for all nodes. Moreover,  for every pair $i,j$ such that $i \neq j$, $H_{P^2}(j \rightarrow i)  \simeq n.$ Thus by Eq. (\ref{uceq}),
\[ \delta_{\rm ss} = \sum_{i=1}^n \sigma_i^2 \frac{1}{n^2} \sum_{j \neq i}  \frac{1}{n} \Theta(n) \simeq  \frac{\sum_{i=1}^n \sigma_i^2}{n}. \] This fact can also be obtained by an easy calculation directly from the definition of $\delta_{\rm ss}$. 

\medskip

\noindent {\bf The circle graph.} Once again, by symmetry we have that $\pi_i = 1/n$ for all nodes. An additional consequence of symmetry is that $H_{\widetilde{P}}(j \rightarrow i) = H_{\widetilde{P}}(i \rightarrow j)$, and so by Eq. (\ref{hsum}) both of these quantities equal half of the resistance between nodes $i$ and $j$. That resistance can be computed by taking
two parallel paths, one with length $|j-i|$ and the other with length $n-|j-i|$; each edge of the path has resistance $O(n)$.  In the worst case, the resitance is quadratic, meaning that we can bound ${\cal H}_{P^2} = O(n^2)$. Thus by Eq. (\ref{uceq}),
\begin{small}
\begin{eqnarray*} \delta_{\rm ss} &  =  & \sum_{i=1}^n \sigma_i^2 \frac{1}{n^2} \sum_{j \neq i} \frac{1}{n} O(n^2) \\ 
&  = & O \left( \sum_{i=1}^n \sigma_i^2 \right). 
\end{eqnarray*}
\end{small}

\medskip

\noindent {\bf The line graph.} On the line graph, we have that the corner nodes have stationary distributions which are $\pi_i \simeq 1/n$. By a standard ``gambler's ruin'' type argument, we have that ${\cal H}_{\widetilde{P}} = O(n^2)$. Thus the calculation is the same as for the ring graph, i.e.,
\begin{eqnarray*} \delta_{\rm ss} &  = & O \left(  \sum_{i=1}^n \sigma_i^2 \right)
\end{eqnarray*} We remark that $\delta_{\rm ss}^{\rm uni}$ has the same scaling, as a consequence of the fact that $\pi_i \simeq 1/n$ for all $i$.

\medskip

\noindent {\bf The star graph.} Let us adopt the convention that node $1$ is the center of the star and nodes $2, \ldots, n$ are the leafs.
We then have that $\pi_1 \simeq 1$ and $\pi_i \simeq 1/n$ for $i=2, \ldots, n$. Furthermore, $H_{P^2}(i \rightarrow 1) \simeq 1$ for $i=2,\ldots, n$ while $H_{P^2}(1 \rightarrow i) \simeq n$ and $H_{P^2}(j \rightarrow i) \simeq n$ for all $i,j$ with $i>1,j>1, i \neq j$. 
Consequently, 
\begin{eqnarray*} \delta_{\rm ss}  &  \simeq & \sigma_1^2  \sum_{j \neq i} \frac{1}{n} 1 + \sum_{i=2}^n \sigma_i^2  \frac{1}{n^2}  \left(1 \cdot n + \sum_{k= 2, \ldots, n, ~~ k \neq i} \frac{1}{n} n \right)   \\
& \simeq & \sigma_1^2 + \frac{ \sigma_2^2 + \cdots + \sigma_n^2}{n}.  
\end{eqnarray*}

 As might be expected, noise at the center vertex contributes an order-of-magnitude more to $\delta_{\rm ss}$  than noise at a leaf vertex with the same variance. We also remark that $\delta_{\rm ss}^{\rm uni}$ is upper by the above scaling since the total number of edges is linear. 

\medskip

\noindent {\bf The two-star graph.} Consider two stars  joined by a link connecting their centers. It is not hard to see that all hitting times in $P^2$ are $\Theta(n)$, with the exception of hitting times from a leaf to its own center, which are $\Theta(1)$ as before. Adopting the conventions of having node $1$ and node $n$ denote the two centers, we have that 
\[ \pi_1 = \pi_n \simeq 1, ~~ \pi_k \simeq \frac{1}{n}, \mbox{ for all } k \neq 1,n.\] 

Thus  \begin{eqnarray*} \delta_{\rm ss} & \simeq & (\sigma_1^2  + \sigma_n^2) ( 1 \cdot n + n \frac{1}{n} \cdot 1 + n \frac{1}{n} n ) \\ && + \sum_{i=2}^{n-1} \sigma_i^2 \frac{1}{n^2}  \left(  2 \cdot 1 \cdot n + \sum_{j \neq i} n \frac{1}{n} \right)   \\ 
& \simeq & n (\sigma_1^2 +  \sigma_n^2 ) + \frac{\sigma_2^2 + \cdots + \sigma_{n-1}^2}{n}.  
\end{eqnarray*} 

It is interesting to compare our results for  the star graph with our results for the two-star graph. While on the star graph, noise at the center vertex contributes $\Theta( n )$ times more to the limiting disagreement than noise at a leaf vertex, on the two-star graph the corresponding factor is $ \Theta \left( n^2 \right)$.
Furthermore, if all $\sigma_i^2$ are positive and bounded away from zero independently of $n$, the disagreement on the two-star graph is $\Theta(n)$ while disagreement on the star graph is $\Theta(1)$. One implication of these comparisons is that the diameter of the graph (which is constant for both the star and the two-star graph) does not determine the order of magnitude of $\delta_{\rm ss}$. 

Finally, we also remark that $\delta_{\rm ss}^{\rm uni}$ is upper by the above scaling since the total number of edges is linear.

\medskip

\noindent {\bf The starry line graph.} We now describe a graph on which $\delta_{\rm ss}$ scales quadratically in the number of nodes $n$ when $\sigma_i^2 = \sigma^2$ for all $i$ -- an order of 
magnitude worse than on all the examples we have considered hitherto. We have not seen this graph described in the literature and we call it the 
{\em starry line graph}.

The construction of the graph is simple. We take a line graph on $n/3$ nodes and two star graphs on $n/3$ nodes (let us assume $n$ is divisible by $3$). We join these three graphs together as follows: we
put an edge between the center of the first star and the left-most vertex of the line and put an edge between the center of the second star and the
right-most vertex of the line. 

We first argue that $\delta_{\rm ss}$ scales at least quadratically on this graph. Indeed, let node $1$ be the center of the first star and let node $n$ be the center of the second star. Considering resistances and using Eq. (\ref{hsum}), we immediately see that   $H_{\widetilde{P}}(1 \rightarrow n)+ H_{\widetilde{P}}(n \rightarrow 1) = \Theta(n^2)$. By symmetry, this implies that both $H_{\widetilde{P}}(1 \rightarrow n)$ and $H_{\widetilde{P}}(n \rightarrow 1)  $ are $\Omega(n^2)$. Since $H_{M^2}(i \rightarrow j) = \Omega(H_M(i \rightarrow j))$, this implies that both $H_{P^2}(1 \rightarrow n), H_{P^2}(n \rightarrow 1)$ are also $\Omega(n^2)$. Since the stationary distribution at both nodes $1$ and $n$ is lower bounded independently of $n$, 
we immediately obtain that $\delta_{\rm ss} \geq \sigma^2 \pi_1^2 \pi_n H_{P^2}(n \rightarrow 1) = \sigma^2 \Omega(n^2)$. 

To get that $\delta_{\rm ss} \simeq n^2$, we argue that the contributions from all other pairs of nodes $i,j$ in Eq. (\ref{samevar}) is not more than $\sigma^2 O(n^2)$. We will use the bound $H_{P^2}(j \rightarrow i) = O(n^2)$ for all $i,j$, which follows from ${\cal H}_{\widetilde{P}}= O(n^2)$ from resistance analysis. Indeed, if neither of $i,j$ is $1$ or $n$, the assertion we need follows since there are $O(n^2)$ such pairs, all with  $\pi_i^2 \pi_j = O(1/n^3)$, so their contribution is $O(n^2 n^3 (1/n^3)) = O(n^2)$. For pairs $i,j$ when one of $i,j$ is $1$ or $n$, we have that there are $O(n)$ such pairs with $\pi_i^2 \pi_j = O(1/n)$, so their contribution is $O(n (1/n) n^2)=O(n^2)$. This concludes the argument.


\medskip

\noindent {\bf The two-dimensional grid.} Let us assume that $n$ is a perfect square. The two-dimensional grid is the graph with the vertex set $\{ (i,j) ~|~ i = 1, \ldots, \sqrt{n}, j = 1, \ldots, \sqrt{n} \}$, and the edge set which is specified by the rule that  $(i_1, j_1)$ and $(i_2, j_2)$ are connected  if and only if $|i_1 - i_2| + |j_1 - j_2| = 1$.  In other words, each node of the 2D grid is labeled by an integer point in the plane, with edges running left, right, up, and down between neighboring points.

By utilizing the formula $\pi_i = d(i)/m$, we immediately have that $\pi_i \simeq 1/n$ for all nodes. A standard  argument (see Theorem 6.1 of \cite{rescomm}) shows that, with unit resistances on each edge, the largest  resistance in the two-dimensional grid is $O( \log n)$. This means that using Eq. (\ref{hsum}) to bound the commute time (which, recall, involves putting a resistor  of resistance $2m=O(n)$ on every edge) we obtain that,
\[ {\cal H}_{\widetilde{P}} = O ( n \log n ), \] and consequently the same bound holds for ${\cal H}_{P^2}$. This implies that 
\begin{eqnarray*} \delta_{\rm ss} & = & \sum_{i=1}^n \sigma_i^2  (n-1) O \left( \frac{1}{n^3} n \log n \right) \\ 
& = &  \left( \sum_{i=1}^n \sigma_i^2 \right) O \left( \frac{\log n}{n} \right). 
\end{eqnarray*} Finally, note that since the degrees on this graph are all $O(1)$, it follows that $\delta_{\rm ss}$ and $\delta_{\rm ss}^{\rm uni}$ are within a constant factor of each other, and consequently $\delta_{\rm ss}^{\rm uni}$ satisfies the same scaling. 

\medskip

\noindent {\bf The $d$-dimensional grid with $d \geq 3$\label{kgrid}.}  We may define the $d$-dimensional grid analogously by associating the nodes with integer points in $\R^d$ and connecting neighbors. According to Theorem 6.1 of \cite{rescomm}, the largest resistance between any two nodes in a $d$-dimensional grid with unit resistors on edges is $\Theta(1/d)$. This becomes $\Theta(n)$ when we put resistors of resistance $2m=\Theta(nd)$ on each each edge. An implication is that ${\cal H}_{\widetilde{P}} = O(n)$.  Since all degrees are within a factor of $2$ of each other, we also have that $1/(2n) \leq \pi_i \leq 2/n$ for all nodes $i$. Putting this together gives 
\[ \delta_{\rm ss} = O \left( \frac{\sum_{i=1}^n \sigma_i^2}{n} \right). \] 
Finally, for the same reason as on the 2D grid, $\delta_{\rm ss}^{\rm uni}$ satisfies the same scaling. 

\medskip

\noindent {\bf The complete binary tree\label{tree}.} It is shown in Section 11.3.1 of \cite{lpw} that for the complete binary tree on $n$ nodes, ${\cal H}_{\widetilde{P}} = O(n \log n)$. Since all degrees are within a factor of $2$ of each other, we have $\pi_i \simeq 1/n$ for all nodes. We thus immediately have the same estimate as for the 2D grid, namely
\[ \delta_{\rm ss} = \left( \sum_{i=1}^n \sigma_i^2 \right) O \left( \frac{\log n}{n} \right). \]  Again since all degrees are within a factor of $2$ of each other, 
$\delta_{\rm ss}^{\rm uni}$ satisfies the same scaling.

\medskip

\noindent {\bf Regular expander graphs.} We first give (one of the) standard definitions of an expander graph. Given a graph $G=(\{1, \ldots, n\}, E)$ and a subset $V' \subset \{1, \ldots, n\}$ we introduce the notation $N(V')$ to denote the set of neighbors of nodes in $V'$, i.e., $N(V) = \{ j ~|~ (i,j) \in E \mbox{ for some } i \in V' \}$. The graph $G$ is called a 
$\alpha$-expander if for every  $V' \subset \{1, \ldots, n\}$ with $|V'| \leq n/2$ we have
$|N(V') - V'| \geq \alpha |V'|$. 

It is Theorem 5.2 in \cite{rescomm} that a regular connected $\alpha$-expander with degree $d$ has resistance at most $O(1/(\alpha^2 d))$ with unit resistors on edges. As a consequence, all commute times in $\widetilde{P}$ are bounded by $O((1/(\alpha^2d)) \cdot d n) = O(n/\alpha^2)$ so that 
\begin{eqnarray*} \delta_{\rm ss} & = & \sum_{i=1}^n \sigma_i^2 \sum_{j \neq i} \frac{1}{n^3} O \left( \frac{n}{\alpha^2} \right) \\ 
& = & \frac{\sum_{i=1}^n \sigma_i^2}{n} ~O \left( \frac{1}{\alpha^2} \right). 
\end{eqnarray*}

Since the graph is regular, $\delta_{\rm ss} = \delta_{\rm ss}^{\rm uni}$ in this case.

\medskip

\noindent {\bf  Dense Erdos-Renyi random graphs.} We next argue that 
\begin{equation} \label{er} \delta_{\rm ss} = O \left(  \frac{\sum_{i=1}^n \sigma_i^2}{n} \right), \end{equation} on an Erdos-Renyi random graph with high probability\footnote{A statement is said to hold with high probability if the probability that it does not hold approaches zero as $n \rightarrow \infty$.}, subject to assumptions  we will spell out shortly. Note that in order to obtain such a result, wewill use the fact that all stationary distribution entries are $\simeq 1/n$ in magnitude and all hitting times are linear. The latter result is apparently available in the literature in \cite{LoweTorres} only for dense Erdos-Renyi random graphs. 

More formally, we consider an undirected Erdos-Renyi random graph on $n$ nodes, meaning that each edge appears independently with a probability of $p_n$. Under the assumption that $(\log n)^{\Theta(\log \log n)}/(n p_n) \rightarrow 0$ as $n \rightarrow \infty$ (this means that the total number of edges in the random graph has expectation that grows slightly faster than $n$, namely faster than $n (\log n)^{\log \log n}$),  it follows from the results of \cite{LoweTorres} that there exists constants $c,C$ such that with high probability we have that for all $i$, \[ cn \leq  \sum_{j=1}^n \pi_j H_{\widetilde{P}} (j \rightarrow i) \leq C n. \] Thus $K(\widetilde{P}) \simeq n$, and therefore $K(P) \simeq n$. Since diagonal dominance of $P$ implies its eigenvalues are nonnegative via Gershgorin circles, we have that $K(P^2) \leq K(P)$ by Eq. (\ref{keig}), and we finally obtain that $K(P^2) = O(n)$ with high probability.  Finally, since $\pi_i = d(i)/2m$ it is quite easy to see that all $\pi$ are on the order of $1/n$ with high probability; formally, we refer the reader to Lemma 3.2 of of \cite{LoweTorres}. We thus have
\[ \delta_{\rm ss} = \sum_{i=1}^n \sigma_i^2 O \left( \frac{1}{n^2} \right)  O(n) = O \left( \frac{\sum_{i=1}^n \sigma_i^2}{n} \right). \]  Finally, $\delta_{\rm ss}^{\rm uni}$ follows the same scaling under these assumptions since with high probability all $\pi_i$ are on the order of $1/n$. 

%
%

\medskip

\noindent {\bf Regular dense graphs.} Let $G$ be a regular graph with degree $d \geq \lfloor n/2 \rfloor$. Then it is Theorem 3.3 in \cite{rescomm} that the largest resistance in such a graph graph with unit resistances on the edges is $O(1/n)$. It we put a resistor of size $2m = O(nd)$ on each edge, the largest resistance becomes $O(d)$. We thus have
\begin{eqnarray*} \delta_{\rm ss} & = &  \sum_{i=1}^n \sigma_i^2 \sum_{j=1}^n \frac{1}{n^3} O(d)   =   O \left( \frac{ \sum_{i=1}^n \sigma_i^2}{n} \right)
\end{eqnarray*}  Once again, because on a regular graph $\delta_{\rm ss} = \delta_{\rm ss}^{\rm uni}$, we have that the same asymptotic holds for $\delta_{\rm ss}^{\rm uni}$.

\smallskip

\noindent {\bf Regular graphs.} We now argue that on any regular graph, $\delta_{\rm ss} =  O \left( \sigma_1^2 + \cdots + \sigma_n^2 \right)$. In particular, this implies that the ring graph 
achieves the worst possible scaling for any  regular graph. At first glance, this might not sound surprising since the ring graph is the sparsest connected regular graph; however, looking at the table at the end of this subsection, we see that there is no clear connection between $\delta_{\rm ss}$ and sparsity. 

This fact is an immediate consequence of the main result of \cite{CFS}, which implies that in a regular graph  ${\cal H}_{\widetilde{P}} = O(n^2)$. Since $\pi_i = 1/n$ for all $i$ due to regularity, we have that 
\[ \delta_{\rm ss} = \sum_{i=1}^n \sigma_i^2 \sum_{j=1}^n \frac{1}{n^3} O(n^2) = O \left( \sum_{i=1}^n \sigma_i^2 \right). \] Moreover, on a regular graph we have that $\delta_{\rm ss} = \delta_{\rm ss}^{\rm uni}$, so that  $\delta_{\rm ss}^{\rm uni}$ satisfies the same upper bound.

\medskip

\noindent {\bf Summary.} We provide a table to summarize all the bounds  for $\delta_{\rm ss}$ on concrete graphs  obtained in the preceeding subsections. 

\begin{center}
\begin{tabular}{|l|c|} 
\hline Graph & $\delta_{\rm ss}$ \\ 
\hline Complete & $\simeq (\sum_{i=1}^n \sigma_i^2 )/n$ \\[1ex]  
\hline Line & $O \left(  \sum_{i=1}^n \sigma_i^2 \right)$ \\[1ex] 
\hline Ring & $O \left( \sum_{i=1}^n \sigma_i^2 \right)$  \\[1ex]   
\hline Star &  $\simeq \sigma_1^2 + (1/n) \sum_{i=2}^n \sigma_i^2 $ \\[1ex]   
\hline Two-star &  $\simeq n (\sigma_1^2 + \sigma_n^2 )+ (1/n) \sum_{i=2}^{n-1} \sigma_i^2 $ \\[1ex]  
\hline Starry line graph & $\simeq \sigma^2 n^2$ when $\sigma_i^2=\sigma$. \\[1ex] 
\hline 2D grid & $(\sum_{i=1}^n \sigma_i^2 ) O ( (\log n)/n)$ \\[1ex]  
\hline kD grid with $k \geq 3$ & $ O (\sum_{i=1}^n \sigma_i^2) /n $ \\[1ex]  
\hline \specialcell{Complete binary\\ tree} & $(\sum_{i=1}^n \sigma_i^2 ) O ( (\log n)/n)$ \\[1ex]  
\hline \specialcell{Regular $\alpha$-expander \\ graphs} & $O(1/\alpha^2) \cdot (\sum_{i=1}^n \sigma_i^2 )/n$ \\[1ex] 
\hline \specialcell{Dense Erdos-Renyi \\ random graphs} & $O (\sum_{i=1}^n \sigma_i^2 )/n$ \\[1ex]  
\hline \specialcell{Regular dense \\ graphs} & $O (\sum_{i=1}^n \sigma_i^2 )/n$ \\[1ex]  
\hline Regular  graphs & $O (\sum_{i=1}^n \sigma_i^2 )$ \\[1ex]  
\hline 
\end{tabular} 
\end{center}

\section{Formation control from noisy relative position measurements\label{symmetric}}

In this section we consider the problem of formation control from noisy relative position measurements, i.e., when each node can measure the (noisy) position of neighboring nodes relative
to itself. We will show that, using Theorem \ref{mainthm}, we can characterize the long-term performance of a class of natural protocols in this settings in terms of the Kemeny constant of an underlying graph.

We begin with a formal statement of the problem.  Our exposition here closely parallels our earlier works \cite{othesis, olinear}. We consider $n$ nodes which start at arbitrary positions $\p_i(0) \in \R^d$. As in the previous sections, there is a graph $(V,E)$, and now the goal of the nodes is to move into a formation which is characterized by certain desired differences along the edges of this graph.

Formally we associate with each edge $(i,j) \in E$ a vector ${\bf r}_{ij} \in \R^d$ known to both nodes $i$ and $j$. A collection of points $\p_1, \ldots, \p_n$ in $\R^d$ are said to be ``in formation'' if 
for all $(i,j) \in E$ we have that $\p_j - \p_i = {\bf r}_{ij}$. In the current section (i.e., in Section \ref{symmetric}), we will be assuming that $G$ is a directed graph with the ``bidirectionality'' property that $(i,j) \in E$ implies $(j,i) \in E$; we will do this so that we may refer to $(i,j)$ and $(j,i)$ as distinct edges of the graph. Naturally, we will also assume $G$ is strongly connected. 

Note  that, given the vectors ${\bf r}_{ij}$, there may not exist a collection of points in formation; that is, some collections of vectors $\{{\bf r}_{ij}, (i,j) \in E\}$ may be thought of as  ``inconsistent.'' For example, unless ${\bf r}_{ij} = -{\bf r}_{ji}$  for all $(i,j) \in E$ the collection $\{{\bf r}_{ij}, (i,j) \in E\}$ will clearly be inconsistent. Moreover, since the property of being in formation is defined through {differences} of position, any translate of a collection of points in formation is itself in formation.  

\begin{figure*} \begin{center}

\begin{tabular}{cc} 
\begin{tikzpicture}[->, thick]
\SetVertexNormal[LineColor=black]
\SetVertexMath

\node (A1) at (0,0) [circle, draw] {$1$};
\node (A2) at (2,2) [circle, draw] {$2$};
\node (A3) at (0,4) [circle, draw] {$3$};
\node (A4) at (-2,2) [circle, draw] {$4$};

\draw (A1) -- (A2) node [midway, fill=white] {${\bf r}_{12} = [1,1]$};
\draw (A2) -- (A1) node [midway, fill=white] {${\bf r}_{12} = [1,1]$};

\draw (A2) -- (A3) node [midway, fill=white] {${\bf r}_{23} = [-1,1]$};
\draw (A3) -- (A2) node [midway, fill=white] {${\bf r}_{23} = [-1,1]$};

\draw (A3) -- (A4) node [midway, fill=white] {${\bf r}_{34} = [-1,-1]$};
\draw (A4) -- (A3) node [midway, fill=white] {${\bf r}_{34} = [-1,-1]$};

\draw (A1) -- (A4) node [midway, fill=white] {${\bf r}_{14} = [1,-1]$};
\draw (A4) -- (A1) node [midway, fill=white] {${\bf r}_{14} = [1,-1]$};

\end{tikzpicture} &~~~~~~~~~~~~~~~~~~ \includegraphics[scale=0.37]{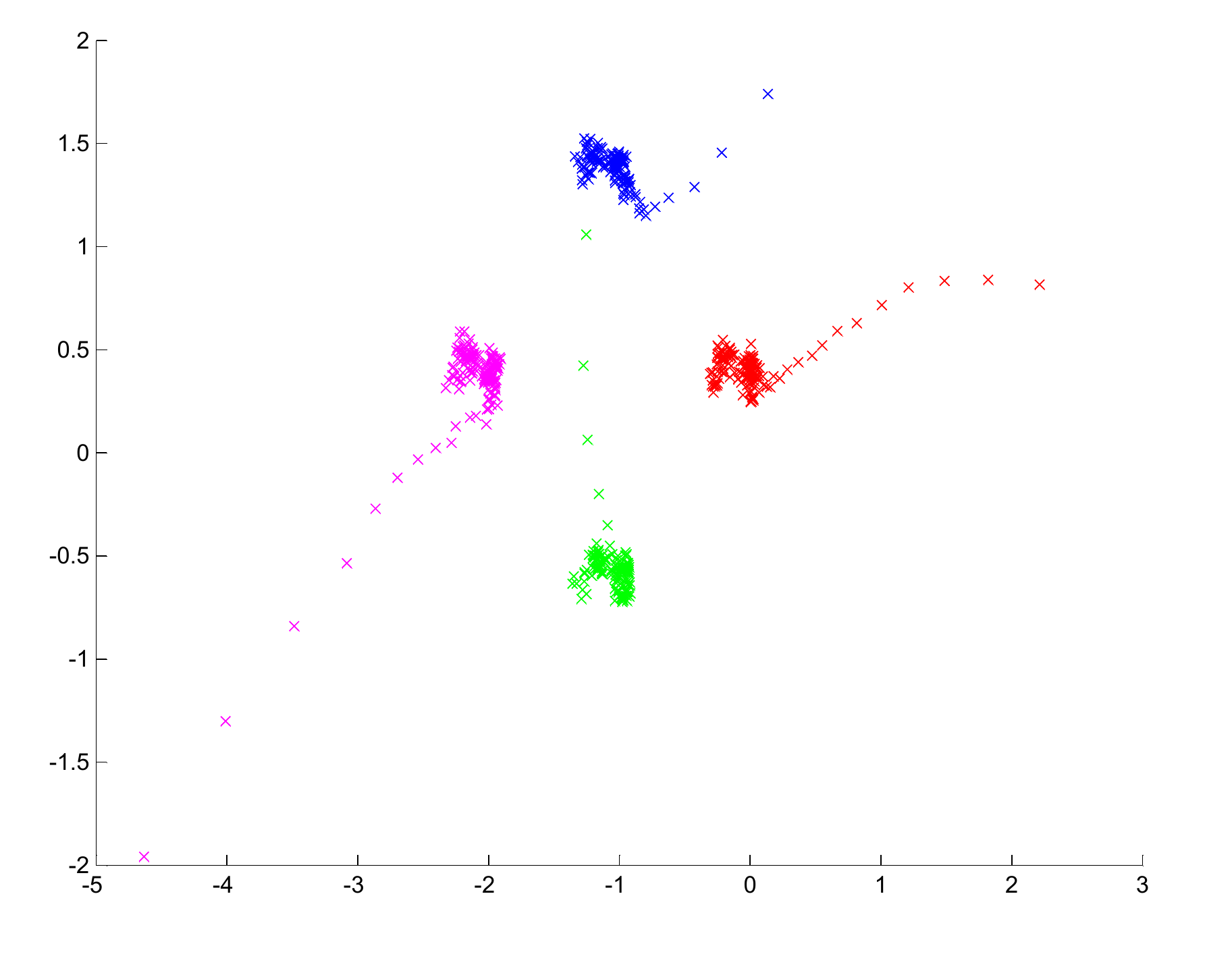} 
\end{tabular}

\caption{The offsets shown on the left side of the figure define a ``ring formation'' with $4$ nodes. On the right, we show the result of simulating Eq. (\ref{formnoise}) on this graph with all the weights $f_{ij}$ equal to $1/9$ starting from random positions. We see that the nodes begin by moving close to the formation and spend the remainder of the time doing essentially a random walk in a neighborhood of the formation. }
\label{ex1}
\end{center} \end{figure*} 
We thus consider the following problem: a collection of nodes would like to repeatedly update their positions so that $\p_1(t), \ldots, \p_n(t)$ approaches some collection of points  in  formation.  We  assume that  node $i$ knows  $\p_j(t)-\p_i(t)$ for all of its neighbors $j$ at every time step $t$ and furthermore we assume a ``first-order'' model in which each node can update its positions from step to step. The protocols we derive for this problem will not assume the presence of a centralized coordinate system common to all the nodes.

A considerable literature has emerged in the past decade spanning many variants of the formation control problem. We make no attempt to survey the vast number of papers that have been published on the topic and refer the interested reader to the surveys \cite{ren, ren2, ofm07}. We stress that the problem setup we have just described is only one possible way to approach the formation control problem; a popular and complementary approach is to consider formations defined by distances $||\p_j - \p_i||_2$ rather than the relative positions $\p_j - \p_i$ (see e.g., \cite{rig1, rig2, rig3, krick}). In terms of the existing literature, our problem setup here is closest to some of the models considered in \cite{krick, ren2, ofm07, fm04, othesis}.


%
 A natural idea is for the nodes to do gradient descent on the potential function $\sum_{(i,j) \in E} || \p_j - \p_i - {\bf r}_{ij}||_2^2$.  This leads to the update rule  \begin{small} \begin{eqnarray} \p_i(t+1) &  = & \p_i(t) + \sum_{j \in N(i)} f_{ij} (\p_j(t) - \p_i(t) - {\bf r}_{ij}),\, \label{formup}
\end{eqnarray} \end{small} where $\{f_{ij}\}$ are positive numbers that, for technical reasons, need to satisfy the step-size condition $\sum_{j \in N(i)} f_{ij} < 1$ for all $i$.  

Note that this update may be implemented in a completely decentralized way as long as node $i$ knows the differences ${\bf p}_j(t) - {\bf p}_i(t)$ and the desired relative positions ${\bf r}_{ij}$. Indeed, the above update allows node $i$ to translate  knowledge of the differences ${\bf p}_j(t) - {\bf p}_i(t)$, which can be measured directly, into knowledge of the difference ${\bf p}_i(t+1) - {\bf p}_i(t)$, which in turn be used to update the current position. In other words, this update may be executed without node $i$ ever knowing what the actual position ${\bf p}_i(t)$ is. 

 It is easy to see that if there exists at least one collection of points in formation, then this control law works in the sense that all $\p_i(t)$ converge and $\p_j(t) - \p_i(t) \rightarrow {\bf r}_{ij}$ for all $(i,j) \in E$ (considerably stronger statements were proved in \cite{fm04, ren}).  For completeness, let us sketch the proof of this simple claim now. If $\overline{\p}_1(t), \ldots, \overline{\p}_n(t)$ is any collection of points in formation, then defining 
\[ {\bf u}_i(t) := \p_i(t) - \overline{\p}_i(t), \] we  have that ${\bf u}_i(t)$ follow the update
\begin{equation} \label{u-up} {\bf u}_i(t+1) = {\bf u}_i(t) + \sum_{j \in N(i)} f_{ij} ( {\bf u}_j(t) - {\bf u}_i(t)). \end{equation} Let $P^{\rm form}$ be the unique stochastic matrix which satisfies $P_{ij}^{\rm form} = f_{ij}$ and let ${\bf u}^j(k)$ be the vector which stacks up the $j$'th entries of the vectors ${\bf u}_1(t), \ldots, {\bf u}_n(k)$. We thus have 
\[ {\bf u}^{j}(k+1) = P^{\rm form} {\bf u}^j(k), ~~~ \mbox{ for all } j = 1, \ldots, d, \] and it is now immediate that all ${\bf u}_i(t)$ approach the same vector. This implies that all $\p_i(t)$ approach positions in formation.

\bigskip

We now turn to the case where the formation control update of Eq. (\ref{formup}) is executed with noise; as we will see, under appropriate assumptions the performance of the (noisy) formation control protocol can be written as the $\delta_{\rm ss}$ of a certain matrix.  Specifically, we will consider the update \begin{footnotesize}
\begin{eqnarray} \label{formnoise} \p_i(t+1) &  = & \p_i(t) + \sum_{j \in N(i)} f_{ij} ( \p_j(t) - \p_i(t) - {\bf r}_{ij} ) + {\bf n}_{i}(t)
\end{eqnarray} \end{footnotesize} The random vector ${\bf n}_i(t)$ can arise if each node executes the motion that updates its position $\p_i(t)$ imprecisely. Although our methods are capable of handling quite general assumptions on the noise vectors ${\bf n}_i(t)$, for simplicity {\em let us assume that  $E[{\bf n}_i(t)] = 0$, $E[ {\bf n}_i(t) {\bf n}_i(t)^T ] = \lambda_i^2 I $ for all $i,t$, and that ${\bf n}_i(t_1)$ and ${\bf n}_j(t_2)$ are independent whenever $t_1 \neq t_2$ or $i \neq j$. }

Of course, once noise is added convergence to a multiple of the formation will not be possible; rather, we will be measuring performance by looking at the asymptotic distance to the closest collection of points in formation. For an illustration, we refer the reader to Figure \ref{ex1} which shows a single run of Eq. (\ref{formnoise}) with four nodes.  As can be read off from the figure, the nodes will move ``towards the formation'' when they are far away from it, but when they are close the noise terms ${\bf n}_i(t)$ effectively preclude the nodes from moving closer and the nodes end up performing random motions in a neighborhood of the formation.

We next formally define the way we will measure the performance of the formation control protocol. Let $\widehat{\p}_1(t), \ldots, \widehat{\p}_n(t)$ be a collection of points in formation whose centroid is the same as the centroid of $\p_1(t), \ldots, \p_n(t)$, i.e., \[ \frac{1}{n} \sum_{i=1}^n  \p_i(t) = \frac{1}{n} \sum_{i=1}^n \widehat{\bf p}_i(t). \] It is easy to see that, as long as there exists a single collection of points in formation, such $\widehat{\bf p}_1(t), \ldots, \widehat{\p}_n(t)$ always exist, and in fact
 $\widehat{\p}_1(t), \ldots, \widehat{\p}_n(t)$ is  closest collection of points in formation to ${\p}_1(t), \ldots, {\p}_n(t)$. Therefore, we will measure the performance of the formation control scheme via the quantity  
\[ {\rm Form}(G, \{f_{ij}\}) := \lim \sup_{t \rightarrow \infty} \frac{1}{n} \sum_{i=1}^n E \left[ ||\p_i(t) - \widehat{\bf p}_i(t)||^2 \right]. \]

In general, obtaining a combinatorial expression for ${\rm Form}(G, \{f_{ij}\})$ is an open problem.  The next proposition describes a solution once again under the additional condition that the weights $\{f_{ij}\}$ are symmetric, i.e., $f_{ij}=f_{ji}$. 

\begin{proposition}[Performance of Formation Control with Symmetric Weights as Steady-State Disagreement] \label{preprop} Let $Q$ be the matrix defined by $Q_{ij} = \lambda_i^2 + \lambda_j^2$. If 
\begin{itemize} \item  There exists at least one collection of points in formation. \item  The underlying graph $G=(V,E)$ is bidirectional and connected. \item The numbers $\{f_{ij}, (i,j) \in E\}$ are positive and  satisfy $\sum_{j \in N(i)} f_{ij} < 1$ for all $i$ and $f_{ij}=f_{ji}$ for all $(i,j) \in E$. \end{itemize} then \begin{small}

\begin{align*} {\rm Form}(G, \{f_{ij}\}) & = d \cdot \delta_{\rm ss} \left(P^{\rm form}, \frac{1}{n} \left( n {\rm Diag}(\lambda_1^2, \ldots, \lambda_n^2) - Q \right. \right. \\ & \left. \left. ~~~~~~~~~~~~~~~~~~~~~~~ + \left( \frac{\sum_{l=1}^n \lambda_l^2}{n} \right) \1 \1^T \right) \right). \end{align*} \end{small}
\label{formprop} 
\end{proposition}

\begin{proof} We proceed by changing variables to
\[ \widehat{{\bf u}}_i(t) = \p_i(t) - \widehat{\p}_i(t).\] Observe that by definition 
\begin{equation} \label{zerosum} \frac{1}{n} \sum_{i=1}^n \widehat{\bf u}_i(t) = 0. \end{equation} Naturally, we also have that \begin{equation} \label{formu} {\rm Form}(G, \{f_{ij}\}) = \lim \sup_{t \rightarrow \infty} \frac{1}{n} \sum_{i=1}^n E \left[|| \widehat{{\bf u}}_i(t) ||_2^2 \right]. \end{equation} We now observe that the symmetry of the weights $\{f_{ij}\}$ as well as the fact that ${\bf r}_{ij}=-{\bf r}_{ji}$  imply that
\[ \frac{1}{n} \sum_{j=1}^n \p_j(t+1) = \frac{1}{n} \sum_{j=1}^n \p_j(t) + \frac{1}{n} \sum_{j=1}^n {\bf n}_j(t), \] which allows us to conclude  that for all $i=1, \ldots, n$,
\[ \widehat{\p}_i(t+1) = \widehat{\p}_i(t) + \frac{1}{n} \sum_{j=1}^n {\bf n}_j(t).\] In turn, this implies that the quantities $\widehat{{\bf u}}_i(t)$ are updated as 

\begin{small}
\begin{eqnarray} \widehat{\bf u}_i(t+1) 
& = &   \widehat{\bf u}_i(t) + \sum_{j \in N(i)} f_{ij} ( \widehat{\bf u}_j(t) - \widehat{\bf u}_i(t)) + {\bf n}_i(t) \nonumber \\ &  & ~~~~~~~ -  \left( \frac{1}{n} \sum_{j=1}^n {\bf n}_j(t) \right). \label{formcons}
\end{eqnarray} \end{small} Now for each $j=1, \ldots, d$, define $\widehat{\bf u}^j(t)$ to stack up the $j$'th components of the vectors $\widehat{\bf u}_1(t), \ldots, \widehat{\bf u}_n(t)$. We then have that Eq. (\ref{formu}) implies 
\begin{equation} \label{formform} {\rm Form}(G, \{f_{ij}\}) = \lim \sup_{t \rightarrow \infty} \sum_{j=1}^d \frac{1}{n} E \left[||\widehat{\bf u}^j(t)||_2^2 \right], \end{equation}
while Eq. (\ref{zerosum}) implies that for all $t$ and $j = 1, \ldots, d$,  \begin{equation} \label{zerosum2} \frac{1}{n} \1^T \widehat{\bf u}^j(t) = 0, \end{equation} and finally Eq. (\ref{formcons}) implies
\begin{equation} \label{compcons} \widehat{\bf u}^j(t+1) = P^{\rm form} \widehat{{\bf u}}^j(t) + {\bf q}^j(t) \end{equation} where the noise vector ${\bf q}^j(t)$ satisfies
 \begin{eqnarray*} E[{\bf q}^j(t)] & = & 0  \\ E[q_k^j(t) q_m^j(t)] & = & - \frac{\lambda_k^2 + \lambda_m^2}{n} + \frac{\sum_{l=1}^n \lambda_l^2}{n^2} ~~~ \mbox{ for all }  k \neq m \\
E[(q_k^j)^2(t)] & = & \lambda_k^2 - 2 \frac{\lambda_k^2}{n} + \frac{\sum_{l=1}^n \lambda_l^2}{n^2}  \mbox{ for all } k. 
\end{eqnarray*} We may summarize these last three equations as  \begin{small}
\begin{eqnarray} \label{noisecov} E \left[{\bf q}^j(t) ({\bf q}^j)^T (t) \right] & = & \frac{1}{n} \left( n {\rm Diag}(\lambda_1^2, \ldots, \lambda_n^2) - Q \right. \nonumber \\ && \left. ~~~~~~ + \frac{1}{n} \left( \sum_{l=1}^n \lambda_l^2 \right) \1 \1^T \right). \end{eqnarray} \end{small}
Equations (\ref{compcons}), (\ref{zerosum2}), (\ref{formform}), (\ref{noisecov}) now immediately imply the proposition. 
\end{proof}

Summarizing, Proposition \ref{formprop} characterizes the performance of a formation control protocol in terms of the $\delta_{\rm ss}$ of an appropriately defined matrix. We can now apply Theorem \ref{mainthm} to obtain a characterization in terms of features of the underlying matrix. For simplicity, let us focus on the case when the noise covariances are the same at each node, i.e.,  \begin{equation} \label{sv} E[{\bf n}_i(t) {\bf n}_i(t)^T] = \lambda^2 I ~~~ \mbox{ for all } i = 1, \ldots, n.\end{equation} In this case, our main result on formation control is as follows. 

\begin{theorem}[Long-term Performance of Noisy Formation Control with Symmetric Weights] Assuming Eq. (\ref{sv}) holds as well as all the assumptions of Proposition \ref{preprop}, we have that \begin{eqnarray*}{\rm Form}(G, \{f_{ij}\}) & = & d \cdot \lambda^2 \frac{K( (P^{\rm form})^2 )}{n}
\end{eqnarray*} \label{formkemeny}
\end{theorem}

\begin{proof} Having already established Proposition \ref{formprop} and Theorem \ref{mainthm}, all that is left is  to combine them. Indeed, if we define 
\[ \Sigma^{\rm form} = \frac{\lambda^2}{n} \left( n I - \1 \1^T\right) \] then Proposition \ref{formprop} for the case of equal-covariances may be 
succintly stated as  
\[ {\rm Form}(G) = d \cdot \delta_{\rm ss} \left( P^{\rm form}, \Sigma^{\rm form} \right). \] Since $P^{\rm form}$ is symmetric, we may apply Eq. (\ref{symmdelta}). 
However, observe that the right-hand side of Eq. (\ref{symmdelta}) is linear in $\Sigma_w$, and plugging in $\Sigma_w = \1 \1^T$ makes the right-hand side of that equation zero. 
Consequently, 
\[ {\rm Form}(G) = d \cdot \delta_{\rm ss} \left( P^{\rm form}, \lambda^2 I \right). \] We  now appeal to 
Eq. (\ref{kemeny}) to complete the proof of this proposition.
\end{proof} 

Thus the long-term performance of formation control is proportional to the Kemeny constant of an underlying matrix. 

We next focus on understanding how the performance of formation control scales with the underlying graph. Of course, there are many possible choices of symmetric $\{f_{ij}\}$ for any given undirected graph $G$. We consider the following choice, which is perhaps the simplest: we set all $f_{ij}$ where $(i,j) \in E$ to some fixed $\epsilon$. In order to satisfy the condition that $\sum_{j \in N(i)} f_{ij}  < 1$ we need to choose $\epsilon$ strictly smaller than the largest degree; to avoid trouble, we therefore choose 
\[ \epsilon = \frac{1}{2 \max_{i} d(i)}. \]  With this choice, ${\rm Form}(G, \{f_{ij}\})$ becomes only a function of the graph $G$, so that we will simply write ${\rm Form}(G)$ henceforth.

We can now use Theorem \ref{formkemeny} to compute the performance of the above-described formation control protocol on various graphs. This requires the computation of hitting times, and since this is something we have done in Section \ref{examples}, we can simply reuse the calculation we have already done (the present choice of coefficients $f_{ij}$ is only a minor modification). We therefore omit an extended discussion and conclude this section with the following list. 

\medskip

\begin{itemize} \item If $G$ is the complete graph, ${\rm Form}(G) \simeq d\lambda^2$. 
\medskip
\item If $G$ is the line graph, ${\rm Form}(G) \simeq dn \lambda^2$. 
\medskip
\item If $G$ is the ring graph, ${\rm Form}(G) \simeq dn \lambda^2$. 
\medskip
\item If $G$ is the 2D grid, ${\rm Form}(G) = d \lambda^2 O ( \log n)$. 
\medskip
\item If $G$ is complete binary tree, ${\rm Form}(G) = d \lambda^2 O ( \log n)$. 
\medskip
\item If $G$ is the 3D grid, ${\rm Form}(G) \simeq d\lambda^2$. 
\medskip
\item If $G$ is the star graph, then ${\rm Form}(G) = O ( d n \lambda^2)$. 
\medskip
\item If $G$ is the two-star graph, then ${\rm Form}(G) = O ( d n \lambda^2)$. 
\medskip
\item If $G$ is a regular $\alpha$-expander, then ${\rm Form}(G) = O( d \lambda^2 / \alpha^2)$. 
\medskip
\item If $G$ is a regular dense graph (recall this means that the degree of each node is at least $\lfloor n/2 \rfloor$), then ${\rm Form}(G) \simeq d \lambda^2$. 
\medskip
\item If $G$ is a regular graph, then ${\rm Form}(G) = O( d n \lambda^2)$. 
\end{itemize}

\section{Simulations\label{simul}}

We now present some simulations  intended to demonstrate how some of the scalings we have derived manifest themselves in some concrete formation control problems. Indeed, a central consequence of our results is that some graphs are  better than others by orders of magnitude. We note that similar observations have been made in the previous literature for a number of concrete graphs; a notable reference is \cite{bamieh1} which considered grids with constant spacing and demonstrated a dramatic difference between the line graph and the 2D and 3D grids. 

We focus here on the star  graph (where ${\rm Form}(G) = O( d n \lambda^2)$) and on the complete binary tree
where ${\rm Form}(G) = O( d \lambda^2 \log n)$. Figures \ref{form1} and \ref{form2} demonstrate the difference between the logarithmic and linear scaling with the number of nodes. In Figure \ref{form1}, we see a single run both protocols with seven nodes; the noise here is rather tiny, $\lambda^2 = 1/2500$, whereas all the relative positions have magnitude $1$ for the star graph and at least one for the binary tree. It might be expected that such a small noise would make relatively little difference, and indeed both formation seem to do reasonably well.

We need a quantitative measure of performance in order to make the last statement precise, which we define as follows. Taking the final positions, $\p_1^{\rm final}, \ldots, \p_n^{\rm final}$ after a given run, we define as in Section \ref{symmetric} the positions $\widehat{\p}_1^{\rm final}, \ldots, \widehat{\p}_n^{\rm final}$ to be positions in formation with the same centroid as $\p_1^{\rm final}, \ldots, \p_n^{\rm final}$. We then define 
\[ {\rm Form}(G, \p_1^{\rm final}, \ldots, \p_n^{\rm final}) := \sum_{i=1}^n \frac{1}{n} \left| \left|\p_i^{\rm final} - \widehat{\p}_i^{\rm final} \right| \right|_2^2. \] 

The quantity ${\rm Form}(G, \p_1^{\rm final}, \ldots, \p_n^{\rm final})$ may be thought as measure of performance: it is  the per-node squared distance to the closest optimal formation. Returning to Figure \ref{form1}, we see that  ${\rm Form}(G, \p_1^{\rm final}, \ldots, \p_n^{\rm final})$ is quite small for both formations. However, as we scale up to $n=127$ in Figure \ref{form2}, we now see that ${\rm Form}(G, \p_1^{\rm final}, \ldots, \p_n^{\rm final}) $ grows much faster on the star formation than on the tree formation, which results in a dramatic difference in performance. In particular, we see that even a tiny noise with $\lambda^2 = 1/2500$ essentially overwhelms the star formation.

\begin{center} \begin{figure*} \begin{center}
\begin{tabular}{cc} 
\includegraphics[scale=0.45]{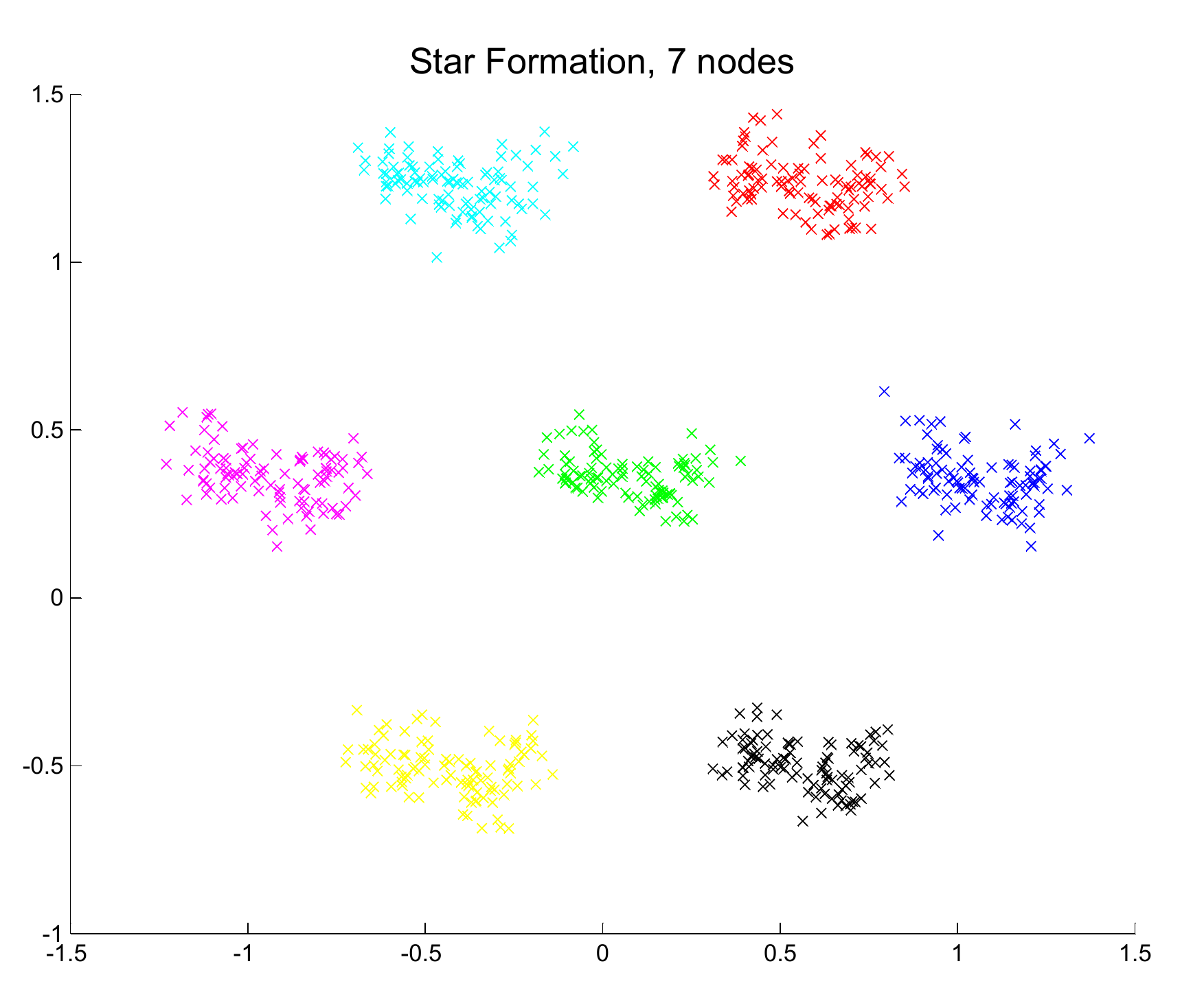} 
&  \includegraphics[scale=0.45]{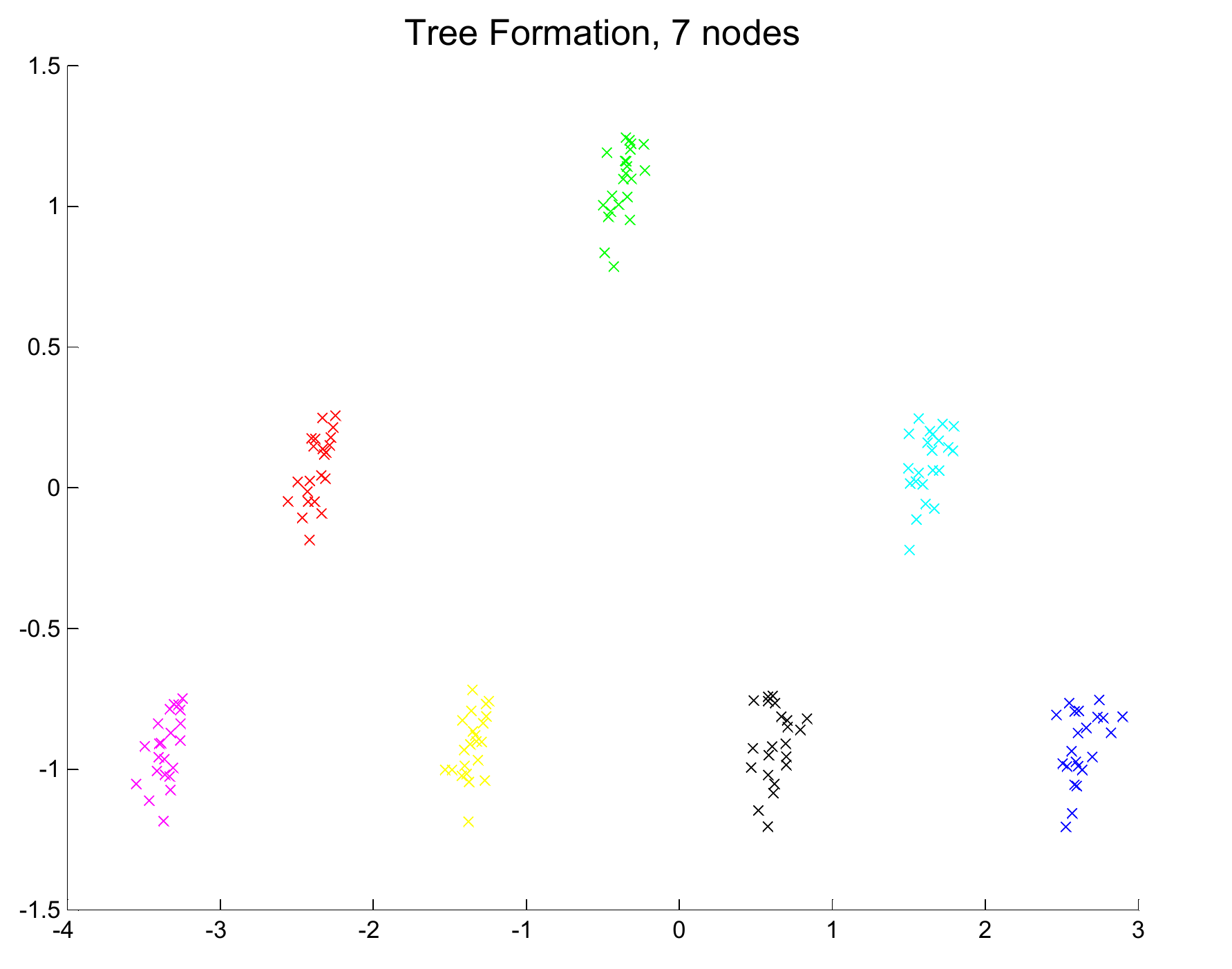} 
\end{tabular} \caption{On the left we show a single run of Eq. (\ref{formnoise}) on a star formation on seven nodes, while on the right we show the same for the tree formation. Both plots show positions from a single run with $w(t)=(1/50) X(t)$ where $X(t)$ are i.i.d. standard Gaussians; each plot shows 22 positions from about 2000 iterations.  Although this is hard to tell with the naked eye, the protocol performs a little better on the star formation here; for the collection of final positions $\p_1^{\rm final}, \ldots, \p_{n}^{\rm final}$, we have that ${\rm Form}(G,\p_1^{\rm final}, \ldots, \p_{n}^{\rm final})  \approx 5 \cdot 10^{-4}$ on the star formation, while ${\rm Form}(G,\p_1^{\rm final}, \ldots, \p_{n}^{\rm final})   \approx 0.001$  on the tree formation.  \label{form1}}\end{center}
\end{figure*} \end{center}

\begin{figure*} \begin{center}
\begin{tabular}{cc} 
\includegraphics[scale=0.45]{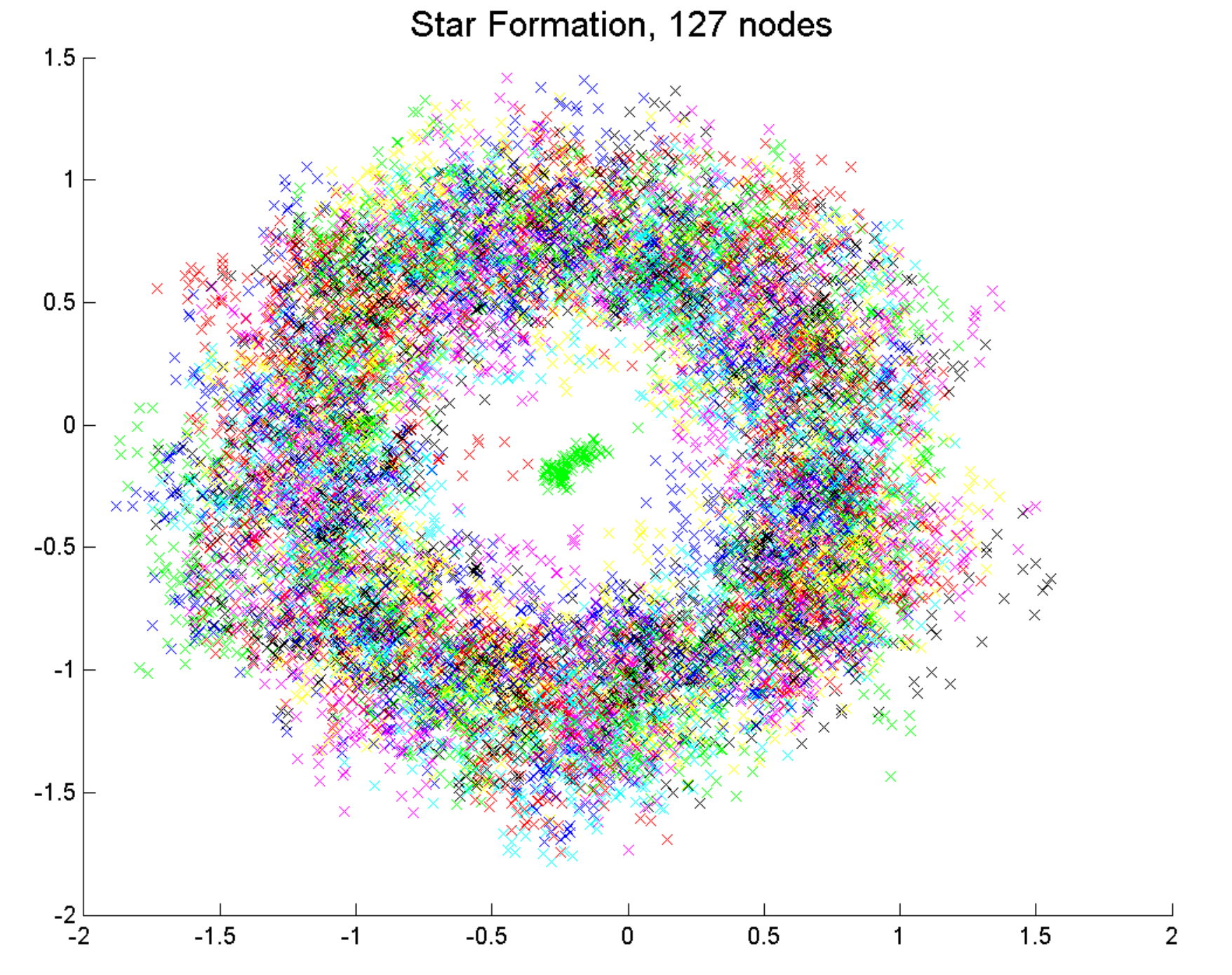} 
&  \includegraphics[scale=0.45]{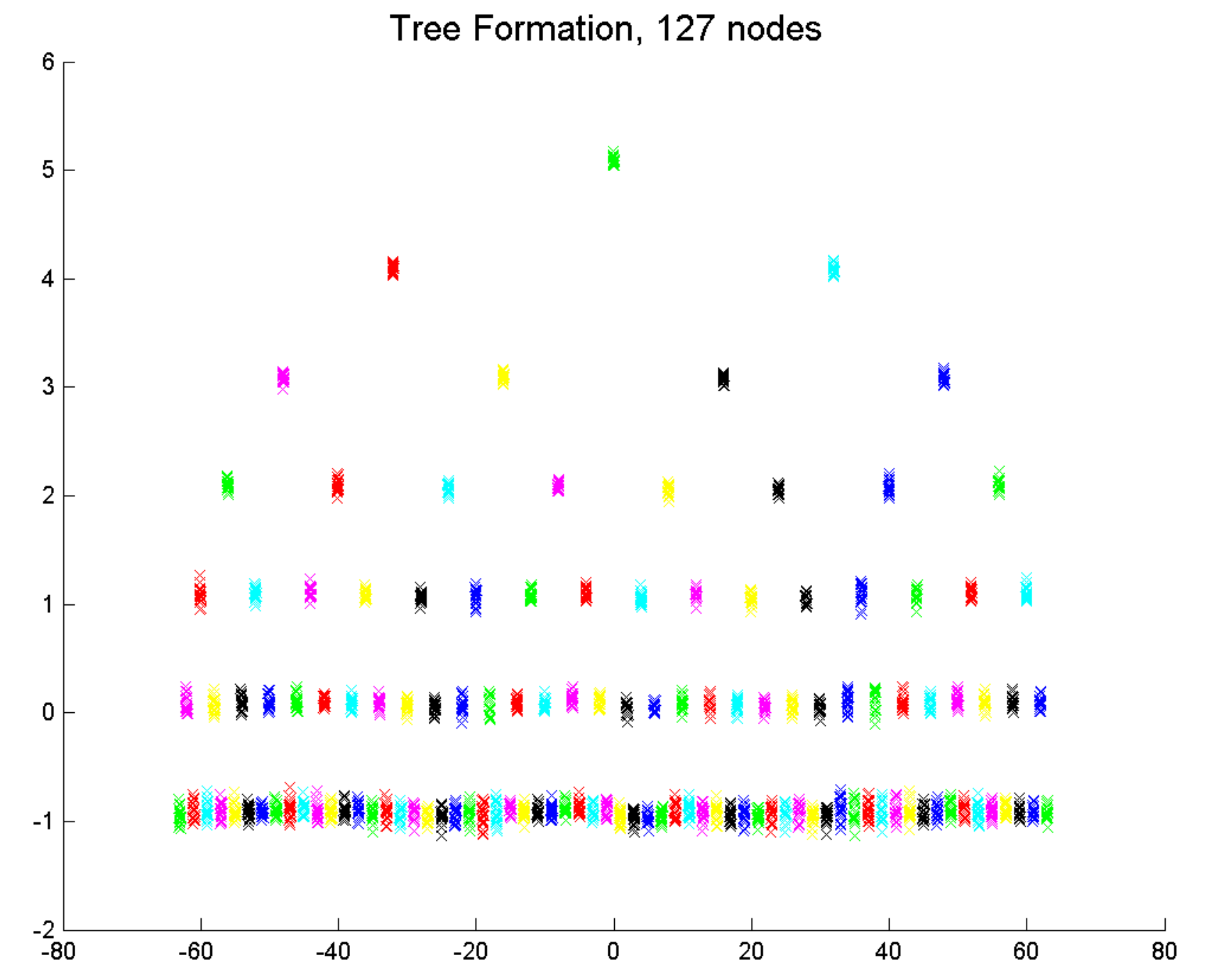} 
\end{tabular} \caption{On the left we show a single run of Eq. (\ref{formnoise}) on a star formation on 127 nodes, while on the right we show the same for the tree formation. Both plots show positios from a single run with $w(t)=(1/50) X(t)$ where $X(t)$ are i.i.d. standard Gaussians; each plot shows 22 positions from about 2000 iterations. We note that the superior appearance of the protocol on the tree formation is not merely due to the increased horizontal spread (see axis labels); in fact, we have that ${\rm Form}(G,\p_1^{\rm final}, \ldots, \p_{n}^{\rm final})  \approx 0.049$ on the star formation, while ${\rm Form}(G,\p_1^{\rm final}, \ldots, \p_{n}^{\rm final})   \approx 0.0049$ (an order of magnitude smaller) on the tree formation.} \label{form2} \end{center}
\end{figure*}

\section{Conclusion\label{conclusion}} The main contributions of this paper are three-fold. First, we have given an explicit expression for the  weighted steady-state disagreement in reversible stochastic linear systems in terms of  stationary distribution and hitting times of appropriate Markov chains. Second, we have given the best currently known bounds for unweighted steady-state disagreement in terms of the same quantities. Finally, we have shown how the Kemeny constant characterizes the performance of a class of noisy formation control protocols. Additionally, we have worked out weighted steady-state disagreement over a number of common graphs. 

An open question is whether similar results might be obtained without the technical assumption of reversibility. Furthermore, the question of obtaining an exact ``combinatorial'  expression for the quantity $\delta_{\rm ss}^{\rm uni}$ is also open. Finally, it is also interesting to wonder how the results we have presented here might be extended to time-varying linear systems. 

More broadly, we wonder whether one can find more connections between probabilistic or combinatorial quantities and the behavior of linear systems. Indeed, we would argue that the past decade of research of distributed control has highlighted the importance of studying linear systems on graphs. Relating classical quantities of interest in control theory, such as stability and noise robustness, to the combinatorial features of the graphs underlying the  system could have a significant repercussions in the control of multi-agent systems.

\end{document}